\documentclass[12pt]{amsart}

\usepackage{geometry}                
\usepackage{graphicx}
\usepackage{amssymb}
\usepackage{xcolor}
\newtheorem{theorem}{Theorem}[section]

\newtheorem{hypothesis}{Hypothesis}[section]

\newtheorem{proposition}[theorem]{Proposition}

\usepackage{multirow}

\usepackage{tikz}
\usepackage{subcaption}
\usepackage{amsthm}
\usepackage[ruled,vlined]{algorithm2e}

\newcommand\restr[2]{{
  \left.\kern-\nulldelimiterspace 
  #1 
  \vphantom{\big|} 
  \right|_{#2} 
  }}

\newcommand{\HN}[1]{\textcolor{black}{#1}}
\newcommand{\reviewerone}[1]{\textcolor{black}{#1}}
\newcommand{\reviewertwo}[1]{\textcolor{black}{#1}}
\providecommand{\keywords}[1]{{#1}}

\title{A stable parareal-like method for the second order wave equation}

\author{Hieu Nguyen} \thanks{Oden Institute for Computational Engineering and Sciences, The University of Texas at Austin, TX 78712, USA. E-mail: hieu@ices.utexas.edu}
\author{Richard Tsai} \thanks{Department of Mathematics and Oden Institute for Computational Engineering and Sciences, The University of Texas at Austin, TX 78712, USA and KTH Royal Institute of Technology, Sweden. E-mail: ytsai@math.utexas.edu}

\begin{document}

\maketitle

\begin{abstract}
  A new parallel-in-time iterative method is proposed for solving the homogeneous second-order wave equation. The new method involves a coarse scale propagator, allowing for larger time steps, and a fine scale propagator which fully resolves the medium using finer spatial grid and uses shorter time steps. The fine scale propagator is run in parallel for short time intervals. The two propagators are coupled in an iterative way that resembles the standard parareal method \cite{parareal-LMT01}. We present a data-driven strategy in which the computed data gathered from each iteration are re-used to stabilize the coupling by minimizing the wave energy residual of the fine and coarse propagated solutions. \HN{Several examples, including a wave speed with discontinuities, are provided to demonstrate the effectiveness of the proposed method.}
\end{abstract}


\section{Introduction}
In this paper,  we will focus on the  initial value problem  of the standard second order wave equation:
\HN{
\begin{align} \label{eq:waveequation} 
    &u_{tt} = c^2(x)\Delta u,\,\,\, x\in[0,1)^d,0\le t <T, \\ 
    &u(x,0) =u_0(x), \nonumber\\
    &u_t(x,0) =p_0(x). \nonumber
\end{align} 
}
\HN{For boundary conditions, we consider either periodic, absorbing boundary conditions or placing a perfectly matched layer around $[0,1)^d$. The wave speed $c(x)$ is given explicitly and independent of the solution. Our objective is to develop a stable parallel-in-time algorithm for \eqref{eq:waveequation}.}

The wave equation is a physical model for seismic wave and electromagnetic wave in certain simplified setups. It is also used as a test case for developing algorithms that are further generalized to more complicated elastic and electromagnetic wave equations.

Time domain decomposition methods for evolution problems has been of increasing interest in the partial differential equation community due to the increasing number of cores available in modern supercomputers. Despite rapid advance in parallel computer architecture, parallelizing the time evolution of the second order wave equation efficiently is still a challenging problem. 
One of the time domain decomposition paradigms is parallel-in-time method. The whole time domain $[0,T)$ is partitioned into subintervals for parallel processing. 
The most relevant algorithm to this paper is 
the parareal method introduced by Lions, Maday and Turinici \cite{parareal-LMT01}. 
The parareal method combines iteratively two propagators, denoted by $\mathcal{F}v$ the fine propagator and by $\mathcal{C}v$ the coarse propagator. They approximate the solution \reviewerone{$v(t_{n+1})$ propagated from $v(t_n)$}. The \reviewerone{approximate solution at parareal iteration $k$, denoted as $v^k_n$,} can be described by
\begin{equation} \label{eq:plainparareal}
\begin{aligned}
    v^{k+1}_{n+1}&=\mathcal{C} v^{k+1}_{n} +\mathcal{F} v^{k}_{n} - \mathcal{C} v^{k}_{n},\\
    v^{1}_0 &= v_0,~~~v^1_{n+1}=\mathcal{C}v^1_n, ~~~k=1,2,\dots,~~n=0,1,\dots,N.
\end{aligned}
\end{equation}
Note that for each $k$, $\mathcal{F}v_n^k$ is computed in parallel. For the second order wave equation under consideration, $v(x,t)$ is a vector corresponding to $[u(x,t),u_t(x,t)]$.

Typically, the coarse propagator runs on coarse grid and is cheaper to compute, while the fine propagator runs on finer grid and is assumed to fully resolve the small scales in the problem. The finer propagator is thus more costly to compute.

In \cite{Bal2005}, it is shown that the parareal method is stable and converges linearly to the serial fine solution if the coarse propagator is smooth and has sufficient dissipation. When certain conditions are met, the parareal method can achieve high fidelity solution within few iterations. Some applications of the parareal method are: plasma turbulence in Tokamak reactor \cite{samaddar2010parallelization,samaddar2019application,samaddar2017temporal}, Navier-Stoke equations \cite{fischer2005parareal,trindade2004parallel,croce2014parallel}, acoustic wave \cite{mercerat2009application}, shallow water \cite{HautOsci13}, chemical kinetic \cite{blouza2011parallel}, molecular dynamics \cite{bylaska2013extending}, reaction wave  \cite{duarte2011parareal}, neutron diffusion \cite{baudron2014parareal,maday2015towards}, lattice Boltzmann equation for laminar flow \cite{randles2014parallel,kreienbuehl2015numerical,randles2014spatio}. 

Speed up of wall-clock time is attained when the coarse propagator can be chosen as a spatial coarsening of the fine propagator  \cite{ruprecht2012explicit,ruprecht2014spatialcoarse} which allows larger coarse time step. Indeed, this coarsening technique provides additional speed up in some applications \cite{lunet2018Turbulence,arteaga2015stencil,kreienbuehl2015numerical} because the coarse propagator has less grid points to compute, provided an appropriate grid restriction and interpolation operator. However as shown in \cite{ruprecht2014spatialcoarse}, considerable coarse grid resolution and accurate interpolation are required in order to make the parareal iteration \eqref{eq:plainparareal} \reviewerone{converge}.

The parareal method tends to suffer from slow convergence or instability when applied to hyperbolic problems. Using an oscillatory dynamical system as an example,  \cite{AKT-2016parareal,ANT-thetaparareal,mikioInfluence18} pointed out that the phase error between the coarse and fine propagators is the reason for the slow convergence. Analogously for advection problems, the authors in \cite{RuprechtWaveChar} observed that numerical dispersion between the solvers \reviewertwo{makes} the parareal method \reviewertwo{converge} from above and hence causes instability. Intuitively, constructive or destructive interference of two overlapping plane waves depends on their relative phase which is sensitive to the frequency, yet the parareal iterative coupling \eqref{eq:plainparareal} is point-wise in space and time.
\newline
\newline
There have been some attempts to modify the classical parareal method in order to address the slow convergence issue. In \cite{dai2013stable}, the fact that solutions to the wave equation live on a submanifold of constant energy is exploited. In that work, the solutions  are projected onto the submanifold to stabilize the parareal iterations. More precisely the algorithm can be presented as 
$$
    u^{k}_n = P[\mathcal{C}u^{k}_{n-1} + (\mathcal{F} u^{k-1}_{n-1} - (\mathcal{C} u^{k-1}_{n-1}))],
$$
where $P$ denotes the projection onto the constant energy submanifold. 
However, the projection is obtained by solving nonlinear equations which can be sensitive to \HN{the} initial guess. 
\newline
\newline
\HN{In the so-called Krylov-subspace enhanced parareal methods \cite{GanderPetcu2008,ruprecht2012explicit}, computed solution data is used to construct projection operators, which is used to modify the coarse propagator. To get the projection operator $P^k$, a set of orthogonal vectors is constructed for the subspace spanned by $(\{u^j_i \}$ for $i=1,2\dots n, j=1,2\dots k-1)$. Let $S^k$ be the matrix whose columns are the orthogonal vectors $s_\ell$, then $P^k=S^k (S^k)^T$. The enhanced parareal algorithm takes the following form:
$$
    u^{k}_n = (\mathcal{C} (I-P^k) u^{k}_{n-1}+\mathcal{F} P^k u^{k}_{n-1}) + \mathcal{F} u^{k-1}_{n-1} - (\mathcal{C} (I-P^k) u^{k-1}_{n-1}+\mathcal{F}P^k u^{k-1}_{n-1}),
$$
where $I$ is the identity. The enhanced coarse propagator corresponds to
$$
    \mathcal{C} (I-P^k)+\mathcal{F} P^k.
$$}
\HN{The fine propagation, $\mathcal{F} s_\ell$ for $s_\ell$ is the orthogonal vector that defines $P^k$, is precomputed and stored. The precomputation incurs an additional computing cost on top of orthogonalization of the data matrices.}

%

\HN{The reduced basis parareal method \cite{Hesthaven2014reduced} develops more efficient ways to construct the basis vectors and extends the approach to solve nonlinear equations.} 

\reviewerone{The convergence and stability of these methods are analyzed and demonstrated by numerical examples of constant wave speed media in one and two dimension. However, in these work, the fine and coarse solvers are assumed to work on the same spatial grids and examples of variable wave speed are not presented. In this paper, we will consider the solvers on different spatial grids and present examples with variable wave speed. We also use the computed data, but its usage is very different from \cite{GanderPetcu2008,ruprecht2012explicit,Hesthaven2014reduced}, see Section \ref{sec:minimizegap}.
}
\newline
\newline
On the other hand, it is known that the slow convergence and instability of the parareal method for hyperbolic problems can be due to some notions of phase errors \cite{ANT-thetaparareal,AKT-2016parareal} and numerical dispersion \cite{RuprechtWaveChar}. 
In \cite{AKT-2016parareal}, effective multiscale parareal schemes relying on elaborate phase correction are proposed for a class of highly oscillatory dynamical systems. 
In \cite{ANT-thetaparareal}, we derived convergence theory for a modified parareal scheme applying to linear systems of ordinary differential equations (ODEs). 
Additionally, in that work, we investigated a few simple strategies of phase correction systematically and showed that appropriate phase correction could enable the resulting scheme to have superior performance. 

In this paper, we propose a new method, based on the idea of $\theta$-parareal scheme \cite{ANT-thetaparareal}. Instead of decomposing the input data as in \cite{GanderPetcu2008,ruprecht2012explicit}, we use the computed data to 
build an operator, formally denoted as $\theta$, that directly brings the coarse solutions, $\mathcal{C}u$, closer to the fine solutions, $\mathcal{F}u$.
In this paper, the $\theta$ operators are constructed by minimizing the residual between the fine and coarse solutions in a semi-norm related to the discrete wave energy. 

\section{Preliminary background}
We briefly review the plain parareal method and its properties. In a context of linear evolutionary problem $\dot{u}(t)=Au(t)$ for $t\in\{0,\Delta t,...N\Delta t = T\}$ and $A:\mathbb{R}\mapsto\mathbb{R}$ linear function, let us denote the fine propagator/solver $\mathcal{F}u_{n} \mapsto u_{n+1}$ and the coarse propagator/solver $\mathcal{C}u_{n}\mapsto u_{n+1}$. Then the plain parareal iteration $k+1$ can be written as a recurrence relation
\begin{equation}\label{eq:plainparareal2}
    u^{k+1}_{n+1}=\mathcal{C} u^{k+1}_{n} +\mathcal{F} u^{k}_{n} - \mathcal{C} u^{k}_{n}.
\end{equation}
Starting solution $k=1$ is the serial coarse solution $u^{k=1}_{n}=\mathcal{C}^{n} u_0$. In addition, by rewriting the recurrence relation \eqref{eq:plainparareal2} in a matrix form and manipulating the inverse of Toeplitz structure, an error estimate \reviewertwo{$e^{k}_n := |u^k_n-u(t_n)|$} is derived in \cite{ANT-thetaparareal}
\begin{equation}
    e^{k+1}_n\leq \| \mathcal{F}- \mathcal{C} \|_{\infty} {\sum}^{n-k-1}_{i=1} \| \mathcal{C}\|^i_{\infty} e^{k}_n.
\end{equation}

The first term on the right hand side is equivalent to the local truncation error of the coarse propagator, assuming the fine solver is an exact one. The summation term is bounded above by $N$ for stable schemes, e.g. $\|\mathcal{C}\|_{\infty} \leq 1$. Above error estimate is equivalent to linear convergence analysis of the parareal method derived in \cite{Bal2005,gander2007analysis}. 

Wall-clock complexity of the parareal algorithm is estimated by
\begin{equation} \label{eq:plaincomplexity}
    C_{parareal} = K \Big(\dfrac{T}{\Delta t}+\dfrac{T}{n_{CPU}\delta t} \Big).
\end{equation}
Comparing to the complexity of the serial fine solver $C_{fine}=T/\delta t$, the parareal algorithm is more effective (from the perspective of total wall-clock computing time) if (i) a large number of computing cores, $n_{CPU}$, are used; (ii) the coarse/fine time stepping ratio is sufficiently large $\Delta t/\delta t \gg 1$; and (iii)  the number of iterations, needed to for the desired accuracy $K$, is small. 

The key objective of this paper is to introduce a data-driven strategy to stabilize and improve the efficiency of the parareal iteration. 

\section{The proposed method}
We propose a scheme that takes the general form:
\begin{equation}
    \mathbf{u}^{k+1}_{n+1}=\theta^k_{n+1}[\tilde{\mathcal{C}} \mathbf{u}^{k+1}_{n}] +\tilde{\mathcal{F}} \mathbf{u}^{k}_{n} - \theta^k_{n+1}[\tilde{\mathcal{C}} \mathbf{u}^{k}_{n}].
\end{equation}
Here, $\mathbf{u}^{k}_{n}$ denotes the solutions computed on the grid, and it has two component blocks, one corresponds to the wave solution $u$ and the other the time derivative $u_t$. In this paper, for readability we shall also use $\dot u$ to denote the time derivative of $u$, i.e. $u_t \equiv \dot{u}$. The coarse and fine propagators, $\mathcal{C}$ and $\mathcal{F}$ will operate on different grids, and additional interpolation and restriction operators are needed for coupling the two propagators. Here we use $\tilde{\mathcal{C}}$ and $\tilde{\mathcal{F}}$ to denote the appropriately defined operations to be described in detail in this section.

A family of operators $\theta^k_{n}[\cdot]$ are constructed such that 
$$\theta^k_{n+1} \approx \tilde{\mathcal{F}}\tilde{\mathcal{C}}^{-1}: \tilde{\mathcal{C}}\mathbf{u} \mapsto \tilde{\mathcal{F}}\mathbf{u}.$$

Clearly, direct calculation of $\tilde{\mathcal{F}}\tilde{\mathcal{C}}^{-1}$ is not practical because it undermines time parallelization of the $\theta$-parareal method. Instead, we seek an effective mapping that has similar property as $\tilde{\mathcal{F}}\tilde{\mathcal{C}}^{-1}$ and is constructed from data computed along the parareal iterations. 
\newline

\subsection{Discretizations and data preparation}
In this paper, we use the uniform Cartesian grids for the spatial domain and uniform stepping in time.
Both the coarse and the fine propagators are defined by 
the standard second order central difference scheme for the spatial derivatives and velocity Verlet for time marching.
The coarse propagator will operate on the coarse grid: $\Delta x\cdot\mathbb{Z}^d\times \Delta t\cdot\mathbb{Z}^+,$ and the fine propagator will operate on the fine grid: $\delta x\cdot\mathbb{Z}^d\times \delta t\cdot \mathbb{Z}^+,$ for $d=1$ or $2$.

Let $u_n \in \mathbb{R}^{N_{\delta x}},U_n \in \mathbb{R}^{N_{\Delta x}}$ denote respectively the solutions computed at time $t_n=n\Delta t_{com},~n=1,2,3,\dots N$ on the fine and coarse grids. \reviewertwo{$N_{\delta x},N_{\Delta x}$ are the number of grid points for the fine grids and the coarse grids respectively.} 

These fine \reviewertwo{grid functions $u\in \mathbb{R}^{N_{\delta x}}$} and coarse grid functions \reviewertwo{$U \in \mathbb{R}^{N_{\Delta x}}$} are coupled by an interpolation $\mathcal{I}:U\mapsto u$ and a restriction $\mathcal{R}:u\mapsto U$. \reviewertwo{The accuracy of the interpolation method will influence the stability of parareal iteration, as discussed in Section \ref{sec:interpinfluence}. Coarse propagator uses point-wise value of the wave speed $c(j\Delta x)$ and does not involve averaging of the wave speed nor homogenization of the wave equation.}
The fine and coarse propagators communicate at $n\Delta t_{com}.$ 
The fine propagator uses the step size $\delta t=\Delta t_{com} /m_{\mathcal{F}}$ and the coarse propagator uses $\Delta t=\Delta t_{com} /m_{\mathcal{C}}$, with $m_{\mathcal{F}}, m_{\mathcal{C}} \in \mathbb{N}$ selected according to $\delta x$ and $\Delta x$ for stability in the respective time stepping. See Figure~\ref{fig:pararealdiscrete}. 
 
\begin{figure}
    \centering
    \includegraphics[width=0.4\linewidth]{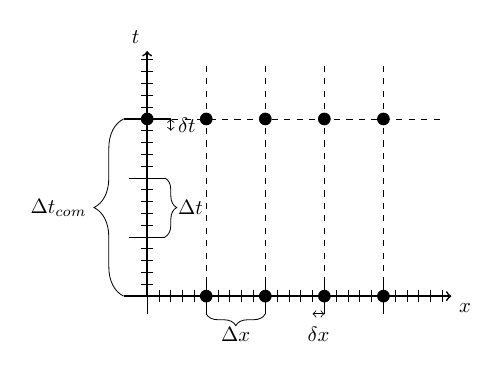}
    \caption{Discretization diagram. Coarse propagator, $\mathcal{C}$ uses spatial grid size $\Delta x$ and temporal step size $\Delta t$. Fine propagator, $\mathcal{F}$ uses spatial grid size $\delta x$ and temporal step size $\delta t$. These propagators communicate at $(j \Delta x, n\Delta t_{com})$ for $j\in\mathbb{Z}, n=1,2,3,\dots,N$.}
    \label{fig:pararealdiscrete}
\end{figure}

Given $[u^{k}_{n-1},\dot{u}^{k}_{n-1}]$ at $t_{n-1}$, the fine and coarse propagators are applied to obtain the solutions to define  
$$[u_n,\dot{u}_n]:=\mathcal{F}[u^{k}_{n-1},\dot{u}^{k}_{n-1}]$$ and  $$[U_n, \dot{U}_n]:=\mathcal{C}[\mathcal{R}u^{k}_{n-1},\mathcal{R}\dot{u}^{k}_{n-1}].$$  
\reviewertwo{For readability, we will write in-line vector $[v, w]$ and full vector
$$\left[\begin{array}{c}
v\\
w
\end{array}\right]$$ interchangeably}. 
These solutions are propagated over a coupling time interval $[t_{n-1}, t_{n})$. These propagators are expected to approximately preserve the wave energy.

Finally, we will quickly describe the data matrices that will be used to construct the operators $\theta^k_n$.
We are interested in using the computed solution data, particularly the gradient of the wavefield $u$ and a weighted momentum of $\dot{u}$. Each column of data matrices is formed by block(s) of the gradients $\nabla U_n$ followed by a block of momentum $\dot{U}_n$ of coarse grid solution at $n$-th coupling time. 
In practice, the gradient operator, $\nabla$, will be replaced by 
some numerical approximation $\nabla_h$.
Then define the data:
\begin{equation} \label{eq:matrixF}
\mathsf{F}:=\left[\begin{array}{cccc}
\nabla_h \mathcal{R} u_{1} & \nabla_h \mathcal{R}u_{2} & \cdots & \nabla_h \mathcal{R}u_{N}\\
c^{-1}\mathcal{R}\dot{u}_{1} & c^{-1}\mathcal{R}\dot{u}_{2} & \cdots & c^{-1}\mathcal{R}\dot{u}_{N}
\end{array}\right],
\end{equation}
\begin{equation} \label{eq:matrixG}
\mathsf{G}:=\left[\begin{array}{cccc}
\nabla_h U_{1} & \nabla_h U_{2} & \cdots & \nabla_h U_{N}\\
c^{-1}\dot{U}_{1} & c^{-1}\dot{U}_{2} & \cdots & c^{-1}\dot{U}_{N}
\end{array}\right].
\end{equation}
Here  and for the rest of the paper, $c^{-1} \dot U_n$ denotes the component-by-component multiplications of  $c^{-1}(x_j)$ and $\dot U_{n} (x_j)$. 
The same convention is used for $c^{-1} \mathcal{R}\dot u_j.$

Now, define the discrete wave energy function as
\begin{equation}\label{eq:discrete-wave-energy}
    E([U_n,\dot{U}_n]) :=\dfrac{1}{2} \sum^{\reviewertwo{N_\Delta x}}_j |\nabla_h U_n(x_j)|^2 \Delta x^d +\dfrac{1}{2} \sum^{\reviewertwo{N_\Delta x}}_j c_j^{-2} |\dot{U}_n (x_j)|^2 \Delta x^d.
\end{equation}
We see that it is equivalent, up to a constant, to the Frobenius norm of the $\mathsf{G}$: 
\begin{equation} \label{eq:partOPP}
    \| \mathsf{G} \|^{2}_{F} = \sum^{N}_{n=1} \Big[ \sum^{\reviewertwo{N_\Delta x}}_j |\nabla_h U_n(x_j)|^2 + \sum^{\reviewertwo{N_\Delta x}}_j c_j^{-2} |\dot{U}_n (x_j)|^2 \Big] \\
    = \dfrac{2}{\Delta x^d} \sum^{N}_{n=1} E([U_n,\dot{U}_n]).
\end{equation}

\subsection{Minimization of coarse-fine solution gaps} \label{sec:minimizegap}
\reviewertwo{For simple plane waves, it is well known that the phase error, not the amplitude difference, between coarse and fine solutions, causes the parareal iteration to converge slowly or diverge \cite{ruprecht2012explicit,mikioInfluence18}. If two plane waves are in phase, parareal style updates can effectively correct the amplitude error. For general wave solutions, it is inconvenient to work with the phase notion defined by the plane wave solutions. Instead, we consider the discrete wave energy semi-norm \eqref{eq:discrete-wave-energy} which is induced by the $\ell^2$ inner-product of the energy component vectors, i.e. the columns of $\mathsf{F},\mathsf{G}$ in \eqref{eq:matrixF} and \eqref{eq:matrixG}. Such inner-product gives us a notion of angle between two wave solutions. The proposed strategy to stabilize the parareal iteration is by minimizing the inner-product between coarse and fine energy component vectors without changing their $\ell^2$ norm. Similar strategies of using wave energy to compare wavefields for wave propagation purposes have been used successfully, for example in seismic imaging \cite{rocha2016acoustic}, wavefield approximation by Gaussian beams \cite{Tanushev-Engquist-Tsai}.}

Denote the $j$-th column of $\mathsf{F}$ and $\mathsf{G}$ by $f_j$ and $g_j$ respectively.
We consider the following optimization problem:
\begin{equation}\label{eq:OPP}
   \min_{\mathsf{\Omega}\in \mathbb{R}^{(d+1)N_{\Delta x} \times (d+1)N_{\Delta x} }} \sum_{j=1}^N ||f_j -\mathsf{\Omega} g_j||_2^2,~~~\text{s.t.}~\mathsf{\Omega}\mathsf{\Omega}^T=\mathsf{\Omega}^T\mathsf{\Omega}=I.
\end{equation}
Recall that the elements in the columns of $\mathsf{F}$ and $\mathsf{G}$ consist of the spatial gradients and weighted time derivatives of the solutions on the respective fine and coarse grid, and that the $\ell^2$ norm corresponds to the discrete wave energy \eqref{eq:discrete-wave-energy}.
Therefore, we look for a unitary matrix so that
the discrete wave energy of the corrected coarse solutions is the same as before correction. 
Intuitively the correction operator aligns the phase (in the above sense) of the coarse solution to fine solution for each $t_n$. It is similar to the local phase-alignment procedure in \cite{AKT-2016parareal} as depicted in Figure \ref{fig:phaseshift}. 
Indeed, from each term in the summation
\begin{equation} \label{eq:fgangle}
    ||f_j-\mathsf{\Omega}g_j||^2_2 = ||f_j||^2_2+||g_j||^2_2-2(f_j,\mathsf{\Omega}g_j),
\end{equation}
the minimization can be interpreted as minimizing the sum of the angles between the columns on the data matrices. Thus, we shall refer to $\mathsf{\Omega}$ as the \emph{phase corrector}.

The minimization problem \eqref{eq:OPP} is equivalent to the "Procrustes Problem" \cite{golub2012matrix}:
\begin{equation} \label{eq:procurstesform}
    \min_{\mathsf{\Omega}\in \mathbb{R}^{(d+1)N_{\Delta x} \times (d+1)N_{\Delta x} }}  \|\mathsf{F} - \mathsf{\Omega} \mathsf{G} \|^{2}_{F},~~~s.t.~~~\mathsf{\Omega}\mathsf{\Omega}^T=\mathsf{\Omega}^T\mathsf{\Omega}=I,
\end{equation}
where $||\cdot||_F$ denotes the Frobenius norm of a matrix. An in-depth review of the Procrustes problem can be found in \cite{gower2004procrustes}. Its variants have been instrumental to multidimensional statistical analysis, rigid body motion simulation, satellite tracking and machine learning \cite{wang2008manifold,ross2008unsupervised,grave2018unsupervised}. 

\begin{figure}
    \centering
    \includegraphics[width=0.4\linewidth]{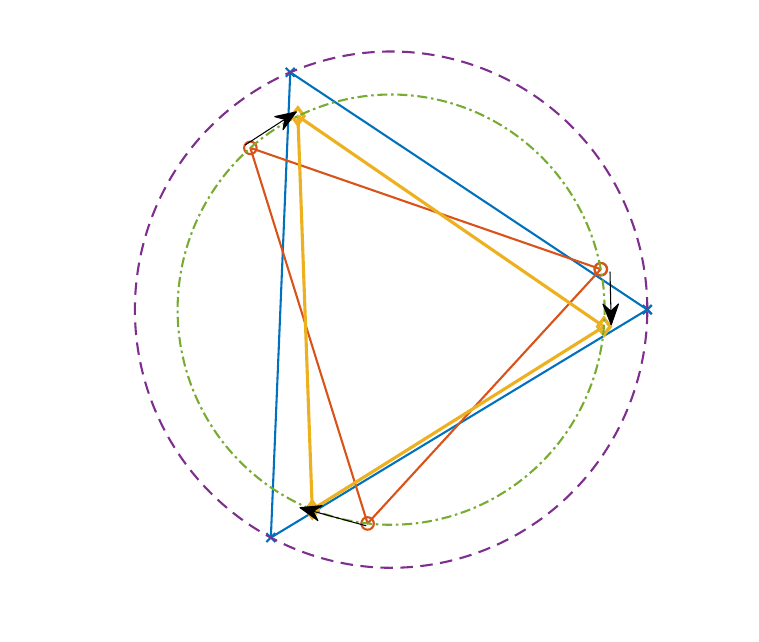}
    \caption{
    Example of Procrustes problem in \eqref{eq:procurstesform}. Blue cross points are reference solution, red circle points represents solution to be aligned blue points. Yellow diamond points are corrected solution of red circle points.}
    \label{fig:phaseshift}
\end{figure}

\subsection{Solution to the optimization problem} \label{sec:solnoptim}
The optimization problem \eqref{eq:OPP} can be solved in a couple of different ways. 
One of them is to use the singular value decomposition (SVD) of the correlation matrix
\begin{equation}\label{eq:correlationmatrixM}
\mathsf{M}:=\mathsf{F}\mathsf{G}^{T}=\sum_{j=1}^{n}\nabla_h \mathcal{R} u_{j}\otimes\nabla_h U_{j}+c^{-1} \mathcal{R} \dot{u}_{j}\otimes c^{-1}\dot{U}_{j}.
\end{equation}

If matrix $\mathsf{M}$ has full rank, the minimizer of \eqref{eq:OPP} is uniquely
\begin{equation}
\mathsf{\Omega}_*=\mathsf{XY}^{T},
\end{equation}
where $\mathsf{X},\mathsf{Y}$ are the left and right singular vectors of $\mathsf{M}=\mathsf{X}\mathsf{\Sigma} \mathsf{Y}^{T}$. 
Correspondingly, the minimum residual is 
\begin{equation*}
    r_{min}^2=\|\mathsf{F}\|^2_F + \|\mathsf{G}\|^2_F - 2~\mathrm{trace}(\mathsf{\Sigma}).
\end{equation*}
Figure \ref{fig:phaseshift} illustrates the Procrustes problem and its solution in a simple setup in $\mathbb{R}^2$.
\newline

\subsubsection{Low rank approximation of $\mathsf{\Omega}_*$} \label{sec:lowrankR}
We now consider a low rank approximation of  $\mathsf{\Omega}_*$ for computational efficiency. 
Since the number of time slices is usually much smaller than the 
number of (coarse) spatial grid nodes, i.e. $N\ll (d+1)N_{\Delta x}$,
we can factorize the data matrices using the reduced QR factorization. Denote the factorizations by
$\mathsf{F}=\mathsf{Q}_F \mathsf{R}_F$ and $\mathsf{G}=\mathsf{Q}_G \mathsf{R}_G$, where  
$$\mathsf{Q}_F,\mathsf{Q}_G \in \mathbb{R}^{(d+1)N_{\Delta x} \times N},~~~~~\mathsf{R}_F,\mathsf{R}_G \in \mathbb{R}^{N \times N}.$$
With the singular value decomposition of the smaller system $\mathsf{R}_F \mathsf{R}_G^T = \mathsf{X}_F \Sigma \mathsf{Y}_G^T$, the correlation matrix can be factored into
$$
    \mathsf{M}=\mathsf{Q}_F \mathsf{X}_F \Sigma \mathsf{Y}_G^T \mathsf{Q}_{G}^{T}.
$$
The last relation shows that
$$\textrm{rank}(\mathsf{M})=\textrm{rank}(\mathsf{R}_F \mathsf{R}_G^T)=\min(\textrm{rank}(\mathsf{F}), \textrm{rank}(\mathsf{G})).$$
Hence we can use the factorization of the smaller $N\times N$ matrix $\mathsf{R}_F \mathsf{R}_G^T$ to obtain
\begin{equation}
    \mathsf{\Omega}_* = (\mathsf{Q}_F \mathsf{X}_F) (\mathsf{Q}_G \mathsf{Y}_G)^T.
\end{equation}
By setting a tolerance to singular values in $\Sigma$, there are $s$ singular values such that $\sigma_s \geq tol$ remained. As the result, we only need to store $s$ number of columns in $\mathsf{Q}_F,\mathsf{Q}_G$, and the truncated phase corrector becomes 

$$ \Omega_{*}:= \Big( \mathsf{Q}_F (:,1:s) \mathsf{X}_F (1:s,1:s) \Big) \Big( \mathsf{Q}_G (:,1:s) \mathsf{Y}_G (1:s,1:s) \Big)^T.$$

\subsubsection{Enriching the phase corrector $\mathsf{\Omega}_*$}
After every parareal iteration, \reviewertwo{more data becomes available}. 
We can use \reviewertwo{this data} to enrich the phase corrector. 
Define
$$
    \mathsf{M}^{k+1} = \mathsf{M}^{k}+\mathsf{F}^k(\mathsf{G}^k)^T.
$$
The singular value decomposition of $M^{k+1}=\Tilde{\mathsf{U}}\Tilde{\mathsf{S}}\Tilde{\mathsf{V}}^T$ can be updated 
using that of $M^k=\mathsf{U}\mathsf{S}\mathsf{V}^T$, see \cite{brand2006fast}. We summarize the update procedure is Algorithm~\ref{updatesvd}.

\begin{algorithm}
$[\Tilde{\mathsf{U}},\Tilde{\mathsf{S}},\Tilde{\mathsf{V}}] \leftarrow \text{UpdateSVD}(\mathsf{U},\mathsf{S},\mathsf{V},\mathsf{F},\mathsf{G},tol): $ \\
    \eIf{$\mathsf{S}$ is empty}{
    $\mathsf{Q}_F \mathsf{R}_F = \texttt{truncatedqr}(\mathsf{F})$\\
    $\mathsf{Q}_G \mathsf{R}_G = \texttt{truncatedqr}(\mathsf{G})$\\
    $\mathsf{X}_{r} \Sigma Y^T_{r} = \texttt{svd}(\mathsf{R}_F \mathsf{R}^T_G)$\\
    $rankM = \texttt{sum}(\texttt{diag}(\Sigma)/\texttt{max}(\texttt{diag}(\Sigma))>tol)$\\
    $\Tilde{U} = \mathsf{Q}_F \mathsf{X}_{r}(:,1:rankM)$\\
    $\Tilde{V} = \mathsf{Q}_G \mathsf{Y}_{r}(:,1:rankM)$\\
    $\Tilde{S} = \Sigma(:,1:rankM)(:,1:rankM)$\\
    }{
    $\mathsf{U}_F = \mathsf{U}^T \mathsf{F} $\\
    $\mathsf{V}_G = \mathsf{V}^T \mathsf{G} $\\
    $\mathsf{Q}_F \mathsf{R}_F = \texttt{truncatedqr}(\mathsf{F}-\mathsf{U}\mathsf{U}_F)$\\
    $\mathsf{Q}_G \mathsf{R}_G = \texttt{truncatedqr}(\mathsf{G}-\mathsf{V}\mathsf{V}_G)$\\
    $\mathsf{H} = [\mathsf{U}_F;\mathsf{R}_F] [\mathsf{V}_G;\mathsf{R}_G]^T + [\mathsf{S} ~~ 0; 0 ~~ 0]$\\
    $X_h \Sigma Y_h = \texttt{svd}(\mathsf{H})$\\
    $rankM = \texttt{sum}(\texttt{diag}(\Sigma)/\texttt{max}(\texttt{diag}(\Sigma))>tol) $\\
    $\Tilde{U} = [\mathsf{U} ~~ \mathsf{Q}_F] \mathsf{X}_{h}(:,1:rankM)$\\
    $\Tilde{V} = [\mathsf{V} ~~ \mathsf{Q}_G] \mathsf{Y}_{h}(:,1:rankM)$\\
    $\Tilde{S} = \Sigma(:,1:rankM)(:,1:rankM)$\\
    }
\caption{Update SVD of the current correlation matrix
$$ \mathsf{M}^{k+1} \equiv \Tilde{\mathsf{U}}\Tilde{\mathsf{S}}\Tilde{\mathsf{V}}^T = \mathsf{U}\mathsf{S}\mathsf{V}^T + \mathsf{F}\mathsf{G}^T $$
where $\mathsf{U}\mathsf{S}\mathsf{V}^T = \mathsf{M}^{k}.$} 
\label{updatesvd}
\end{algorithm}

\subsection{Reconstruction of wavefield from the gradient} \label{sec:funcreconstruct}
After correcting the energy components, i.e. the gradients and the weighted time derivatives, of the coarse solutions, it is necessary to reconstruct the wavefield pair from the corrected energy components. In other words, we denote
$[q,p]$ as the corrected energy components of a wavefield pair $[w,\dot{w}]$
\begin{equation}
\left[\begin{array}{c} q \\
p
\end{array}\right]\equiv
\mathsf{\Omega}_{*} \Lambda \left[\begin{array}{c}
w\\
\dot w
\end{array}\right]
:=
\mathsf{\Omega}_{*} \left[\begin{array}{c}
\nabla_h w\\
c^{-1}\dot w
\end{array}\right],
\end{equation}
where the mapping $\Lambda:[w,\dot{w}]\mapsto[\nabla_h w, c^{-1} \dot{w}]$ takes function to wave energy components. Then we want to deduce the corrected wavefield pair $[v, \dot v]$ such that
$$
    \left[\begin{array}{c} \nabla v\\
c^{-1}\dot{v}
\end{array}\right]
\simeq \left[\begin{array}{c} q\\
p
\end{array}\right].
$$ 
It is straightforward to find the latter component $ \dot{v} = c p$. For the displacement component $v$, we use the spectral property of differentiation 
$\texttt{fft}\{\nabla v\}=i\boldsymbol{\xi}\texttt{fft}\{v\}$
to recover its the Fourier modes as follow 
\begin{equation} \label{eq:gradient2coordinate}
\texttt{fft}\{v\}=\begin{cases}
-i(\boldsymbol{\xi}\cdot\texttt{fft}\{ q \})|\boldsymbol{\xi}|^{-2} & \text{for }|\boldsymbol{\xi}| \neq0,\\
\sum^{N_{\Delta x}}_{j} w(x_j) & \text{for } |\boldsymbol{\xi}|=0.
\end{cases}
\end{equation}
We denote this mapping from energy component to wavefield component as $\Lambda^\dag:[\nabla v, c^{-1}\dot{v}]\mapsto [v,\dot{v}]$. In particular, when the gradient is approximated by Fourier method, this reconstruction is an identity.
\begin{proposition}
    Suppose the gradient of function $v(x)$ is estimated by spectral method $\nabla_h v \equiv \texttt{ifft} \{i\boldsymbol{\xi} \texttt{fft} \{v \} \}$, then 
    \begin{equation}
        \Lambda^\dag \Lambda \left[\begin{array}{c} v\\
        \dot{v}
        \end{array}\right]
        = \Lambda^\dag 
        \left[\begin{array}{c} \texttt{ifft} \{i\boldsymbol{\xi} \texttt{fft} \{v \} \} \\
        c^{-1}\dot{v}
        \end{array}\right]
        =\left[\begin{array}{c} v\\
        \dot{v}
        \end{array}\right].
    \end{equation}
\end{proposition}
\begin{proof}
Let 
\begin{align*}
    \left[\begin{array}{c} w\\
        \dot{w}
        \end{array}\right] &= \Lambda^\dag \Lambda
    \left[\begin{array}{c} v\\
    \dot{v}
    \end{array}\right]
\end{align*}
Since $\Lambda$ maps function to energy components we have
\begin{align*}
    \Lambda^\dag \Lambda \left[\begin{array}{c} v\\
        \dot{v}
        \end{array}\right] &= \Lambda^\dag
    \left[\begin{array}{c} \nabla_h v\\
    c^{-1} \dot{v}
    \end{array}\right]
\end{align*}
By construction of $\Lambda^\dag$, for nonzero wavenumber $|\boldsymbol{\xi}|\neq 0$
\begin{align*}
    \texttt{fft} \{w \} & = -i\boldsymbol{\xi} \cdot \texttt{fft} \{\nabla_h v\} |\boldsymbol{\xi}|^{-2}.
\end{align*}
Here the gradient is approximated using spectral method then
\begin{align} \label{eq:spectraldiff}
    \texttt{fft} \{w \} & = -i\boldsymbol{\xi} \cdot \{ i \boldsymbol{\xi} \texttt{fft} \{v\} \} |\boldsymbol{\xi}|^{-2}
    \\
    & = \texttt{fft}  \{v\}. \nonumber
\end{align}
And for zero wavenumber $|\boldsymbol{\xi} |=0$,
$\texttt{fft} \{w\} = \sum_j v(x_j) = \texttt{fft} \{v\}.$
Thus, $w = v$ while the second energy component $\dot{w} = c c^{-1}\dot{v}=\dot{v}.$ This concludes that the mapping $\Lambda^\dag \Lambda$ is equal to identity.
\end{proof}
If the gradient is approximated by a central finite difference of $2m$-order instead of the spectral method, for one dimensional setting equation \eqref{eq:spectraldiff} in the proof above  becomes 
\begin{align*}
    \texttt{fft} \{w \} & = -i\xi [ i \Delta x^{-1} \sum^m_{j=1} (e^{i j \xi \Delta x} - e^{- i j \xi \Delta x} )\beta_j ~ \texttt{fft} \{v\} ] |\xi|^{-2}
    \\
    & = 2 (\xi \Delta x)^{-1} \sum^m_{j=1} \sin(j \xi \Delta x) \beta_j  \texttt{fft} \{v\},
\end{align*}
where $\beta_j$ are appropriate coefficients of the difference stencil. When the spatial grid is small enough $\xi \Delta x \ll 1,$ above expression is approximately 
\begin{align*}
    \texttt{fft} \{w \} & = 2 (\xi \Delta x)^{-1} \sum^m_{j=1} (j \xi \Delta x - \dfrac{1}{3!}(j \xi \Delta x)^3 + \mathcal{O}(j \xi \Delta x)^5) \beta_j ~ \texttt{fft} \{v\}
    \\
    & = 2 \sum^m_{j=1} (j - \dfrac{1}{3!} j^3 (\xi \Delta x)^2 + \mathcal{O}j^5 (\xi \Delta x)^4) \beta_j ~ \texttt{fft} \{v\}.
\end{align*}
Particularly for the second order central difference $m=1$, we would have $\beta_1 =1/2$, then
\begin{align*}
    \texttt{fft} \{w \} & = \text{sinc}(\xi \Delta x)  \texttt{fft} \{v\},
\end{align*}
which says that $|\texttt{fft} \{w \}| \leq |\texttt{fft} \{v\}|$ because $\text{sinc}(\xi \Delta x)\leq 1$. 

In practice, we observe that the algorithm does not require spectral approximation of the gradient, but instead $\| \Lambda^\dag \Lambda\|_2 \leq 1$ is necessary for stability of the method. When central finite difference is utilized, it is well known that the modified wavenumber is less than $|\boldsymbol{\xi}|$, hence central difference satisfies the requirement $\| \Lambda^\dag \Lambda\|_2 \leq 1$. Algorithm \ref{reconsfuncgrad} summarizes above procedure.

\begin{algorithm}
$ w \leftarrow \text{grad2func} (\nabla_h w, \sum w) $: \\
    $\hat{p}=\texttt{fft}(\nabla_h w)$\\
    
    $\hat{q}(|\boldsymbol{\xi}| \neq 0) = -i\boldsymbol{\xi}\cdot\hat{p}|\boldsymbol{\xi}|^{-2}$\\
    $\hat{q} (|\boldsymbol{\xi}|=0)=\sum w$\\
    $ w = \texttt{ifft}(\hat{q})$
\caption{Reconstruct function from the gradient.} 
\label{reconsfuncgrad}
\end{algorithm}

\subsection{The proposed algorithm}
\begin{algorithm}
\SetAlgoLined
Initialization: 
$[u^{k=1}_{n},\dot{u}^{k=1}_{n}] =\mathcal{I}\mathcal{C}[\mathcal{R}u^{k=1}_{n-1},\mathcal{R}\dot{u}^{k=1}_{n-1}]$\\
$\mathsf{\Sigma} = [~]$, $\mathsf{X} = [~]$, $\mathsf{Y} = [~]$ \\
\While{tolerance not meet and $k\leq K$}
{
    par\For{$n=2\rightarrow N$}{
            $[v_n,\dot{v}_n]=\mathcal{F}[u^{k}_{n-1},\dot{u}^{k}_{n-1}]$\\
            $[U_n,\dot{U}_n]=\mathcal{C}[\mathcal{R}u^{k}_{n-1},\mathcal{R}\dot{u}^{k}_{n-1}]$
        }
        $ $\\
    Solve the orthogonal Procrustes problem:\\
    $\mathsf{F}=[\nabla_h \mathcal{R} v_n,\mathcal{R} c^{-1} \dot{v}_n]$\\
    $\mathsf{G}=[\nabla_h U_n,c^{-1} \dot{U}_n]$\\
    $[\mathsf{X}, \mathsf{\Sigma},\mathsf{Y}]=\texttt{UpdateSVD}(\mathsf{X},\mathsf{\Sigma},\mathsf{Y},\mathsf{F},\mathsf{G},tol) $\\
    $ $\\
    \For{$n=2\rightarrow N$}{
    $[w,\dot{w}]=\mathcal{C} [\mathcal{R}u^{k+1}_{n-1},\mathcal{R}\dot{u}^{k+1}_{n-1}]$\\
    $[q,p]= \mathsf{X}\mathsf{Y}^T[\nabla_h w,c^{-1}\dot{w}] $\\
    $[\Tilde{q},\Tilde{p}]= \mathsf{X}\mathsf{Y}^T[\nabla_h U_n,c^{-1}\dot{U}_n]$\\
    Reconstruct function from gradient: \\
    $q_1 = \texttt{grad2func}(q,\sum w) $ \\
    $\Tilde{q}_1 = \texttt{grad2func}(\Tilde{q},\sum U_n) $ \\
    Update next time step: \\
    $[u^{k+1}_{n},\dot{u}^{k+1}_{n}]=[v_{n},\dot{v}_{n}] + \mathcal{I}\Big([\texttt{real}(q_1 - \Tilde{q}_1),c(p-\Tilde{p}]) \Big)$\\
    }
    $k=k+1$\\
}
\caption{The proposed algorithm.} 
\label{NewPararealAlgo}
\end{algorithm}

The proposed algorithm couples the fine and the coarse propagators at
times $n\Delta t_{com}, n=1,2,\dots,N$ over the fine grid (the spatial grid that the fine solutions are defined). However, it is important to note that \textit{the phase corrections are applied on the coarse grid}. If the two grids are not identical, an interpolation is needed. We denote the interpolation operator by $\mathcal{I}$. 
Furthermore, denote the mappings between the wavefield $[v,\dot v]$ and its energy components $[\nabla v, c^{-1}\dot v]$ by $\Lambda:[v,\dot{v}]\mapsto [\nabla v, c^{-1} \dot{v}]$ and  $\Lambda^\dag:[\nabla v, c^{-1} \dot{v}]\mapsto [v,\dot{v}]$.
With these notations, the $\theta$ operator after $k$ iterations can be written as
\begin{equation*}
    {\theta^{k}}[v,\dot{v}] = \mathcal{I} \Lambda^\dag \mathsf{\Omega}^k_{*} \Lambda [v,\dot{v}].
\end{equation*}
Here  we use $\mathsf{\Omega}^k_{*}$ to denote the
phase corrector derived from the data matrix $\mathsf{M}^{k}$.

Finally, our new algorithm can be written compactly as in $\theta$-parareal form  
\begin{align} \label{eq:thetarelation}
\left[\begin{array}{c}
u_{n+1}^{k+1}\\
\dot{u}_{n+1}^{k+1}
\end{array}\right] = 
\theta^{k} \mathcal{C}
\left[\begin{array}{c}
\mathcal{R}u^{k+1}_{n} \\
\mathcal{R}\dot{u}^{k+1}_{n}
\end{array} \right] +
\mathcal{F}\left[\begin{array}{c}
u^{k}_{n} \\ 
\dot{u}^{k}_{n}\end{array} \right] - 
{\theta^{k}}\mathcal{C} \left[ \begin{array}{c}
\mathcal{R}u^{k}_{n} \\
\mathcal{R}\dot{u}^{k}_{n} \end{array} \right].
\end{align}
\newline
Algorithm \ref{NewPararealAlgo} describes the new scheme in a pseudo-code form with more details.

\reviewerone{Similar to the Krylov subspace method \cite{GanderPetcu2008,Hesthaven2014reduced}, our method requires orthogonalization of data matrices, but they are formed in a different way. In this paper, the data matrices are the multiplication of the wave energy components of the fine data and the coarse data as in equation \eqref{eq:correlationmatrixM}. Then the phase correctors are constructed from the singular value decomposition of the data matrices. 
In \cite{GanderPetcu2008,Hesthaven2014reduced}, the data matrices, consisting of computed solutions, are orthogonalized to form projection operators. In contrast, our phase correctors are not projections, but they effectively induce translation of the coarse solutions on constant energy submanifolds. }

\section{Complexity Analysis}

There are three parts to our implementation: parallel fine propagator computation, construction of $\mathsf{\Omega}^*$ and the serial coarse  updates. We assume that (i) no spatial domain decomposition, i.e. whole domain on a single core, (ii) standard QR complexity, i.e. no multithreading, (iii) communication between nodes and other tasks negligible. 

In each iteration, the wall clock complexity for the parallel fine and coarse computations is in the order of
$$
    \dfrac{1}{n_{CPU}}\Big( \dfrac{T}{\delta t} N_{\delta x} + \dfrac{T}{\Delta t} N_{\Delta x} \Big)(d+1)
$$
where $n_{CPU}$ is the number of cores, $N_{\delta x},N_{\Delta x}$ are respectively the total number of fine and coarse grid points.
The complexity of serial coarse update in an iteration is
$$
    \dfrac{T}{\Delta t} (d+1)N_{\Delta x}.
$$
The complexity of standard QR factorization for constructing $\mathsf{\Omega}$ is
$$
    (d+1)N_{\Delta x}  N^2.
$$
Therefore, the total complexity is
\begin{equation}
    K(d+1)\Big(\dfrac{T}{n_{CPU}\delta t} N_{\delta x} + \dfrac{T}{n_{CPU}\Delta t} N_{\Delta x} +
    \dfrac{T}{\Delta t} N_{\Delta x} + N_{\Delta x} N^2  \Big),
\end{equation}
where $K$ is the number of iterations. In this set up, the speed up over a serial fine computation is 
\begin{equation}
    \Big[ K (\dfrac{1}{n_{CPU}} + (\dfrac{1}{n_{CPU}}+1)\dfrac{\delta t N_{\Delta x}}{\Delta t N_{\delta x}} + \dfrac{N_{\Delta x} N \delta t}{N_{\delta x} \Delta t_{com}})  \Big]^{-1}.
\end{equation}
Additionally, we have coarse/fine time step ratio $\Delta t / \delta t = m_t$, which implies $\Delta t_{com}/\delta t\ge m_t$, and their corresponding the mesh ratio is $\Delta x/ \delta x = m_s$. Hence the theoretical speed up is

\begin{equation}\label{eq:newspeedup}
    \dfrac{1}{K} \min \{\mathcal{O} (n_{CPU}), \mathcal{O} (m_s^d m_t),\mathcal{O} \big(\dfrac{m_s^d m_t}{N} \big)\}.
\end{equation}

\reviewerone{We note that the third term in above speed up is derived from the classical $N^2$ complexity for QR factorization (of matrices of fixed number of rows). The speed up analysis \eqref{eq:newspeedup} presents the worst-case asymptotics as \textit{N approaches infinity}. In practice, we observe that QR factorization has sub-quadratic scaling when multithreading, ubiquitous in modern computers, is enabled. However, to our knowledge, speed up analysis of QR in a multithreading environment is not straightforward. To illustrate the effectiveness of multithreading in computing QR factorization, consider random matrices with fixed $100,000$ number of rows and vary the number of columns in a way relevant to the paper. The computing time is presented in Figure \ref{fig:qrtime}. The computing time roughly grows as the power of $1.5$ of the number of columns, rather than quadratically according to the classical QR complexity. We also see that having 68 threads speed up the computation by more than a factor of 10. 
Finally from numerical experiments, the QR step in our algorithm takes relatively small amount of time compared to other components, see Section \ref{sec:timing}.}

\begin{figure}
    \centering
    \includegraphics[width=0.5\linewidth]{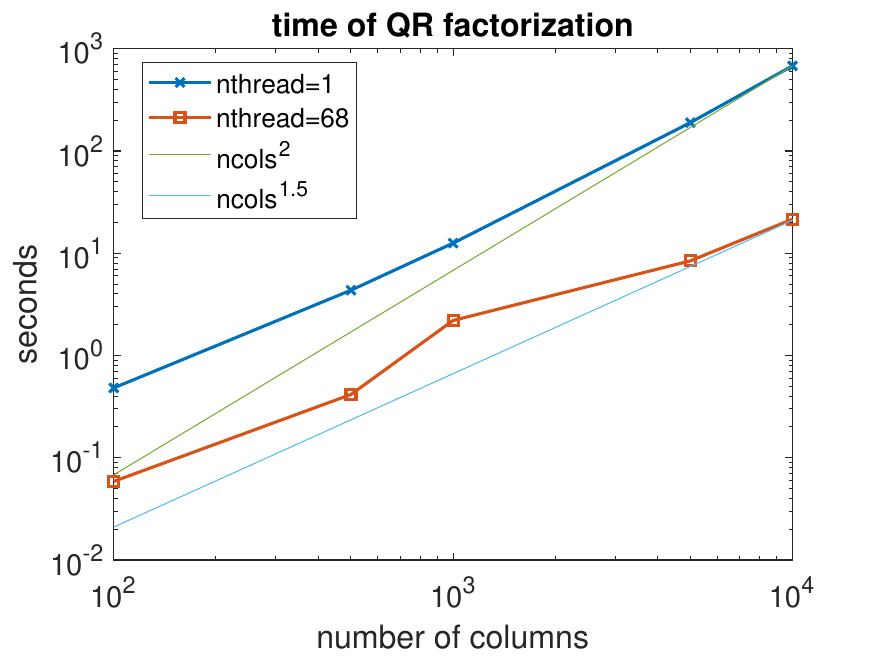}
    \caption{Computing time of QR factorization as function of number of columns.}
    \label{fig:qrtime}
\end{figure}

\section{Stability and convergence}
In this section, we will derive some estimates that show the stability and the convergence of algorithm \ref{NewPararealAlgo} under certain assumptions.
We measure the difference, in the discrete energy semi-norm on the coarse grid, between the serially computed fine solution and the iterated solution.

Consider energy components of parareal iterated solution restricted on the coarse grid
$$
\left[\begin{array}{c} 
    \nabla_h U^k_{n+1} \\
    \dfrac{1}{c} \dot{U}^k_{n+1}
\end{array} \right]
\equiv 
\Lambda \mathcal{R} \left[\begin{array}{c} 
    u^k_{n+1} \\
    \dot{u}^k_{n+1}
\end{array} \right].
$$
Its parareal iterative coupling is expressed as equation \eqref{eq:thetarelation}
\begin{equation}\label{eq:energy-parareal-form}
\left[\begin{array}{c} 
    \nabla_h U^k_{n+1} \\
    \dfrac{1}{c} \dot{U}^k_{n+1}
\end{array} \right]
= 
\Lambda \mathcal{R}
\left(
\theta^{k-1} \mathcal{C}
\left[\begin{array}{c}
\mathcal{R}u^{k}_{n} \\
\mathcal{R}\dot{u}^{k}_{n}
\end{array} \right] +
\mathcal{F}\left[\begin{array}{c}
u^{k-1}_{n} \\ 
\dot{u}^{k-1}_{n}\end{array} \right] - 
{\theta^{k-1}}\mathcal{C} \left[ \begin{array}{c}
\mathcal{R}u^{k-1}_{n} \\
\mathcal{R}\dot{u}^{k-1}_{n} \end{array} \right]
\right).
\end{equation}
Recall that  ${\theta}[v,\dot{v}] := \mathcal{I} \Lambda^\dag \mathsf{\Omega} \Lambda [v,\dot{v}]$, so 
$$
\Lambda \mathcal{R}\theta^{k-1}= \Lambda \mathcal{R} \mathcal{I} \Lambda^\dag \mathsf{\Omega}^k_* \Lambda.
$$
Since the restriction operator takes point wise values, it cancels action of the interpolation $\mathcal{R}\mathcal{I}=1$. So equation \eqref{eq:energy-parareal-form} becomes
\begin{equation} \label{eq:newpararealwavecomponent}
\left[\begin{array}{c} 
    \nabla_h U^k_{n+1} \\
    \dfrac{1}{c} \dot{U}^k_{n+1}
\end{array} \right]
= 
\Lambda \Lambda^\dag \mathsf{\Omega}^k_*\Lambda \mathcal{C}
\left[\begin{array}{c}
U^{k}_{n} \\
\dot{U}^{k}_{n}
\end{array} \right] +
\Lambda \mathcal{R} \mathcal{F}\left[\begin{array}{c}
u^{k-1}_{n} \\ 
\dot{u}^{k-1}_{n}\end{array} \right] - 
\Lambda \Lambda^\dag \mathsf{\Omega}^*_k\Lambda \mathcal{C} \left[ \begin{array}{c}
U^{k-1}_{n} \\
\dot{U}^{k-1}_{n} \end{array} \right].
\end{equation}
Let us denote the square root of wave energy as $\mathcal{E}([U,\dot{U}]) := \sqrt{E([U,\dot{U}])}$, where $E$ is defined in \eqref{eq:discrete-wave-energy}. 
Thus,
$$
    \mathcal{E}([U^{k}_n,\dot{U}^k_n]) = \| \left[\begin{array}{c} 
    \nabla_h U^k_{n} \\
    \dfrac{1}{c} \dot{U}^k_{n}
    \end{array} \right] \|_2.
    $$

\begin{theorem} \label{theorem:energystability}
    Suppose that 
    \begin{enumerate}
        \item the coarse propagator $\mathcal{C}$ satisfies, for some $\epsilon>0$;
        $$ \mathcal{E}(\mathcal{C}[U,\dot{U}]) \leq \mathcal{E}([U,\dot{U}]) + \epsilon;$$ 
        \item the residual of the energy minimization problem is bounded uniformly for $k=1,2,\dots$: 
        $$ \|\mathsf{F}^k-\mathsf{\Omega}^k_*\mathsf{G}^k \|_F \le \eta$$ 
        where $\mathsf{F}^k,\mathsf{G}^k$ are data matrices in \eqref{eq:matrixF},\eqref{eq:matrixG} gathered in the first $k$ iterations;
        \item $\| \Lambda^\dag\Lambda \|_2 \leq 1$, and $\|1-\Lambda^\dag\Lambda \|_2<\lambda \reviewertwo{\ll} 1/N$. 
     \end{enumerate}
     Then 
    \begin{equation}
        \max_{n\leq N}\mathcal{E}([U^{k}_n,\dot{U}^k_n]) \leq \frac{1}{1-\lambda N}\Big( \mathcal{E}([U_0,\dot{U}_0]) + (N+1)\epsilon + C \eta \Big),
    \end{equation}
    where $C$ is a norm equivalence constant between $\ell_{2,1}$ \reviewertwo{norm (sum of $\ell_2$ norm of columns)} and Frobenius \reviewertwo{norm}.
\end{theorem}

\begin{proof}
    \reviewertwo{Consider the square root of wave energy of \eqref{eq:newpararealwavecomponent}
    \begin{align*}
    \mathcal{E}([U^{k}_{n+1},\dot{U}^k_{n+1}]) & = 
    \| \left[\begin{array}{c} 
    \nabla_h U^k_{n+1} \\
    \dfrac{1}{c} \dot{U}^k_{n+1}
    \end{array} \right] \|_2
    \\ & = 
    \| \Lambda \Lambda^\dag \mathsf{\Omega}^k_*\Lambda \mathcal{C}
    \left[\begin{array}{c}
    U^{k}_{n} \\
    \dot{U}^{k}_{n}
    \end{array} \right] +
    \Lambda \mathcal{R} \mathcal{F}\left[\begin{array}{c}
    u^{k-1}_{n} \\ 
    \dot{u}^{k-1}_{n}\end{array} \right] - 
    \Lambda \Lambda^\dag \mathsf{\Omega}^*_k\Lambda \mathcal{C} \left[ \begin{array}{c}
    U^{k-1}_{n} \\
    \dot{U}^{k-1}_{n} \end{array} \right] \|_2.
    \end{align*}}
    We apply triangle inequality to obtain
    \begin{align*}
    \mathcal{E}([U^{k}_{n+1},\dot{U}^k_{n+1}]) & \leq \| \Lambda \Lambda^\dag \mathsf{\Omega}\Lambda \mathcal{C}
    \left[\begin{array}{c}
    U^{k}_{n} \\
    \dot{U}^{k}_{n}
    \end{array} \right] \|_2 + 
    \| \Lambda \mathcal{R} \mathcal{F}\left[\begin{array}{c}
    u^{k-1}_{n} \\ 
    \dot{u}^{k-1}_{n}\end{array} \right] - 
    \Lambda \Lambda^\dag \mathsf{\Omega}\Lambda \mathcal{C} \left[ \begin{array}{c}
    U^{k-1}_{n} \\
    \dot{U}^{k-1}_{n} \end{array} \right] \|_2
    \\
    & \leq 
    \| \Lambda \Lambda^\dag \|_2\,\,\| \mathsf{\Omega}\|_2\,\, \|\Lambda \mathcal{C}
    \left[\begin{array}{c}
    U^{k}_{n} \\
    \dot{U}^{k}_{n}
    \end{array} \right] \|_2 + 
    \| \Lambda \mathcal{R} \mathcal{F}\left[\begin{array}{c}
    u^{k-1}_{n} \\ 
    \dot{u}^{k-1}_{n}\end{array} \right] - 
    \Lambda \Lambda^\dag \mathsf{\Omega}\Lambda \mathcal{C} \left[ \begin{array}{c}
    U^{k-1}_{n} \\
    \dot{U}^{k-1}_{n} \end{array} \right] \|_2.
    \end{align*}
    By construction, $\|\mathsf{\Omega}\|_2 = 1$, and by the hypotheses that $\|\Lambda \Lambda^\dag \|_2\leq 1$  and energy bound of the coarse propagator,
    $$\|\Lambda \mathcal{C} [U^k_{n},\dot{U}^k_{n}] \|_2=\mathcal{E}(\mathcal{C}[U^k_{n},\dot{U}^k_{n}])\leq \mathcal{E}([U^k_{n},\dot{U}^k_{n}]) + \epsilon,$$ 
    we have
    \begin{align*}
    \mathcal{E}([U^{k}_{n+1},\dot{U}^k_{n+1}]) & \leq \mathcal{E}([U^k_{n},\dot{U}^k_{n}]) + \epsilon +
    \| \Lambda \mathcal{R} \mathcal{F}\left[\begin{array}{c}
    u^{k-1}_{n} \\ 
    \dot{u}^{k-1}_{n}\end{array} \right] - 
    \Lambda \Lambda^\dag \mathsf{\Omega}^k_*\Lambda \mathcal{C} \left[ \begin{array}{c}
    U^{k-1}_{n} \\
    \dot{U}^{k-1}_{n} \end{array} \right] \|_2
    \\
    & \leq
    \mathcal{E}([U^k_{n},\dot{U}^k_{n}]) + \epsilon +
    \| \Lambda \mathcal{R} \mathcal{F}\left[\begin{array}{c}
    u^{k-1}_{n} \\ 
    \dot{u}^{k-1}_{n}\end{array} \right] - 
    \mathsf{\Omega}^k_*\Lambda \mathcal{C} \left[ \begin{array}{c}
    U^{k-1}_{n} \\
    \dot{U}^{k-1}_{n} \end{array} \right] -
    \Lambda \Lambda^\dag \mathsf{\Omega}^k_*\Lambda \mathcal{C} \left[ \begin{array}{c}
    U^{k-1}_{n} \\
    \dot{U}^{k-1}_{n} \end{array} \right] \|_2.
    \end{align*}
    Seeing the third term as part of the energy minimization problem in \eqref{eq:partOPP}, 
    \begin{align*}
    \mathcal{E}([U^{k}_{n+1},\dot{U}^k_{n+1}]) & \leq \mathcal{E}([U^k_{n},\dot{U}^k_{n}]) + \epsilon + \| f_{n+1} - \mathsf{\Omega}^k_* g_{n+1} \|_2 
    + \|(1-\Lambda \Lambda^\dag)
     \mathsf{\Omega}^k_*\Lambda \mathcal{C} \left[ \begin{array}{c}
    U^{k-1}_{n} \\
    \dot{U}^{k-1}_{n} \end{array} \right] \|_2
    \\
    & \leq
    \mathcal{E}([U^k_{n},\dot{U}^k_{n}]) + \epsilon +
    \| f_{n+1} - \mathsf{\Omega}^k_* g_{n+1} \|_2
    + \|1-\Lambda \Lambda^\dag \|_2 \Big( \mathcal{E}([U^{k-1}_{n},\dot{U}^{k-1}_{n}]) + \epsilon \Big)
    \\
    & \leq
    \mathcal{E}([U^k_{n},\dot{U}^k_{n}]) + \epsilon +
    \| f_{n+1} - \mathsf{\Omega}^k_* g_{n+1} \|_2
    + \lambda \Big( \mathcal{E}([U^{k-1}_{n},\dot{U}^{k-1}_{n}]) + \epsilon \Big)
    \\
    & \leq
    \mathcal{E}([U_{0},\dot{U}_{0}]) + (n+1)\epsilon +
    \sum^n_{j=1}\| f_j - \mathsf{\Omega}^k_* g_j \|_2
    + \sum^{n}_{j=0} \lambda \Big( \mathcal{E}([U^{k-1}_{j},\dot{U}^{k-1}_{j}]) + \epsilon \Big)
    \\
    & \leq
    \mathcal{E}([U_{0},\dot{U}_{0}]) + (n+1)\epsilon +
    C \eta
    + \lambda n \Big( \max_{j\leq N} \mathcal{E}([U^{k-1}_{j},\dot{U}^{k-1}_{j}]) + \epsilon \Big).
    \end{align*}
    As the above relation also holds for $\max_{j\leq N} \mathcal{E}([U^k_j,\dot{U}^k_j])$, therefore, 
    \begin{align*}
        \max_{j\leq N}\mathcal{E}([U^{k}_j,\dot{U}^k_j]) & \leq \mathcal{E}([U_{0},\dot{U}_{0}]) + (N+1)\epsilon +
    C\eta
    + \lambda N \Big( \max_{j\leq N}\mathcal{E}([U^{k-1}_j,\dot{U}^{k-1}_j]) + \epsilon \Big)
    \\ 
    & =
    \lambda N \max_{j\leq N} \mathcal{E}([U^{k-1}_j,\dot{U}^{k-1}_j])  + \Big( \lambda N\epsilon +
    \mathcal{E}([U_{0},\dot{U}_{0}]) + (N+1)\epsilon +
    \reviewertwo{C} \eta \Big).
    \end{align*}
    Applying the discrete Gr{\"o}nwall inequality \reviewertwo{\cite{holte2009discrete}} \reviewertwo{on index $k$}  we get
    \begin{align*}
    \max_{j\leq N}\mathcal{E}([U^{k}_j,\dot{U}^k_j]) & \leq
    (\lambda N)^{k-1} \max_{j\leq N} \mathcal{E}([U^1_j,\dot{U}^1_j])
    \\
    & + \Big( \lambda N\epsilon + \mathcal{E}([U_{0},\dot{U}_{0}]) + (N+1)\epsilon +
    \reviewertwo{C} \eta \Big) \sum^{k-1}_{l=0} (\lambda N)^l.
    \end{align*}
    \reviewertwo{By the assumption $\lambda N \ll 1$,}
    \begin{align*}
        \max_{j\leq N}\mathcal{E}([U^{k}_j,\dot{U}^k_j]) & \leq
    \Big( \lambda N\epsilon + \mathcal{E}([U_{0},\dot{U}_{0}]) + (N+1)\epsilon +
    C\eta \Big) \dfrac{1}{1-\lambda N}.
    \end{align*}
\end{proof}

Next, we will show that, under some hypotheses, the proposed method converges to the solutions computed by applying the fine propagator serially. \reviewertwo{The hypotheses involve Lipschitz smoothness of the phase corrector, which implies the minimization problem \eqref{eq:OPP} is solved with sufficient accuracy.}
We shall use the following notation for those reference solutions:
\begin{equation}
    [u(t_n), \dot{u}(t_n)]:= \mathcal{F}^n [u_0,\dot{u}_0],~~~n=1,2,\dots,N.
\end{equation}
We measure the overall error on the fine grid as the square root of the difference in the discrete wave energy:
\begin{equation}
    \mathcal{E}^k_n 
    := \| \Lambda [u^k_{n} - u(t_{n}), \dot{u}^k_{n} - \dot{u}(t_{n})] \|_2.
\end{equation}

\begin{hypothesis} \label{hypo:lipschitz}
(i) The phase corrected coarse solution is Lipschitz continuous in energy 
\begin{align*}
    \| \Lambda \theta \mathcal{C} \mathcal{R} \left[\begin{array}{c}
    v\\
    \dot{v}
    \end{array} \right]
    -
    \Lambda \theta \mathcal{C} \mathcal{R} \left[\begin{array}{c}
    w\\
    \dot{w}
    \end{array} \right]
    \|_2
    & = 
    \| \Lambda \mathcal{I} \Lambda^\dag \mathsf{\Omega} \Lambda \mathcal{C} \mathcal{R} \left[\begin{array}{c}
    v\\
    \dot{v}
    \end{array} \right] -
    \Lambda \mathcal{I} \Lambda^\dag \mathsf{\Omega} \Lambda \mathcal{C} \mathcal{R} \left[\begin{array}{c}
    w\\
    \dot{w}
    \end{array} \right]
    \|_2
    \\
    & \leq
    (1+\epsilon_{\mathcal{I}\mathcal{R}})(1+\epsilon_{\Lambda^\dag \mathsf{\Omega} \Lambda \mathcal{C}})  \| \Lambda \left[\begin{array}{c}
    v-w\\
    \dot{v} - \dot{w}
    \end{array} \right] \|_2.
\end{align*}
Let $\epsilon_{\theta}$ denote the overall perturbation 
\begin{align} \label{eq:hypo1}
    \| \Lambda \theta \mathcal{C} \mathcal{R} \left[\begin{array}{c}
    v\\
    \dot{v}
    \end{array} \right]
    -
    \Lambda \theta \mathcal{C} \mathcal{R} \left[\begin{array}{c}
    w\\
    \dot{w}
    \end{array} \right]
    \|_2
    \leq
    (1+\epsilon_{\theta})
    \| \Lambda  \left[\begin{array}{c}
    v-w\\
    \dot{v}-\dot{w}
    \end{array} \right]  \|_2.
\end{align}

(ii) The energy error between fine and corrected coarse operators is Lipschitz continuous
\begin{align} \label{eq:hypo2}
    \| (\Lambda \mathcal{F} - \Lambda \theta \mathcal{C} \mathcal{R})
    \left[\begin{array}{c}
    v\\
    \dot{v}
    \end{array} \right]
    - 
    (\Lambda \mathcal{F} - \Lambda \theta \mathcal{C} \mathcal{R})
    \left[\begin{array}{c}
    w\\
    \dot{w}
    \end{array} \right]
    \|_2
    \leq
    \kappa
    \| \Lambda  \left[\begin{array}{c}
    v-w\\
    \dot{v}-\dot{w}
    \end{array} \right]  \|_2.
\end{align}
\end{hypothesis}

\begin{theorem} \label{theorem:errorconvergence}
    Suppose that the fine and corrected coarse operators satisfy Hypothesis~ \eqref{eq:hypo1} and \eqref{eq:hypo2}. 
    Then,
    \begin{equation}
        \max_{j\leq N}\mathcal{E}^k_j  \leq \kappa \dfrac{(1+\epsilon_{\theta})^{N}-1}{\epsilon_{\theta}}
    \max_{j\leq N}~\mathcal{E}^{k-1}_{j}.
    \end{equation}

\end{theorem}

\begin{proof}
In the following expansion of the parareal iteration, the superscript $k$ in $\theta^k$ are dropped for brevity

\begin{align*}
\left[\begin{array}{c}
u_{n+1}^{k}\\
\dot{u}_{n+1}^{k}
\end{array}\right] & = 
\theta \mathcal{C}\mathcal{R}
\left[\begin{array}{c}
u^{k}_{n} \\
\dot{u}^{k}_{n}
\end{array} \right] +
\mathcal{F}\left[\begin{array}{c}
u^{k-1}_{n} \\ 
\dot{u}^{k-1}_{n}\end{array} \right] - 
{\theta}\mathcal{C}\mathcal{R} \left[ \begin{array}{c}
u^{k-1}_{n} \\
\dot{u}^{k-1}_{n} \end{array} \right] 
\\
& = 
\theta \mathcal{C}\mathcal{R} \Big(
\theta \mathcal{C}\mathcal{R}
\left[\begin{array}{c}
u^{k}_{n-1} \\
\dot{u}^{k}_{n-1}
\end{array} \right] +
\mathcal{F}\left[\begin{array}{c}
u^{k-1}_{n-1} \\ 
\dot{u}^{k-1}_{n-1}\end{array} \right] - 
{\theta}\mathcal{C}\mathcal{R}
\left[ \begin{array}{c}
u^{k-1}_{n-1} \\
\dot{u}^{k-1}_{n-1} \end{array} \right]
\Big)
\\
& +
\mathcal{F}\left[\begin{array}{c}
u^{k-1}_{n} \\ 
\dot{u}^{k-1}_{n}\end{array} \right] - 
{\theta}\mathcal{C}\mathcal{R}
\left[ \begin{array}{c}
u^{k-1}_{n} \\
\dot{u}^{k-1}_{n} \end{array} \right]
\\
& = 
(\theta \mathcal{C}\mathcal{R})^{n+1}\left[\begin{array}{c}
u_{0} \\ 
\dot{u}_{0}\end{array} \right]
+ (\theta \mathcal{C}\mathcal{R})^{n}
\Big( \mathcal{F}\left[\begin{array}{c}
u_{0} \\ 
\dot{u}_{0}\end{array} \right] - 
{\theta}\mathcal{C}\mathcal{R}
\left[ \begin{array}{c}
u_{0} \\
\dot{u}_{0} \end{array} \right] \Big) 
\\ 
& + 
(\theta \mathcal{C}\mathcal{R})^{n-1}
\Big( \mathcal{F}\left[\begin{array}{c}
u^{k-1}_{1} \\ 
\dot{u}^{k-1}_{1}\end{array} \right] - 
{\theta}\mathcal{C}\mathcal{R}
\left[ \begin{array}{c}
u^{k-1}_{1} \\
\dot{u}^{k-1}_{1} \end{array} \right] \Big) \dots
\\
& +
(\theta \mathcal{C}\mathcal{R})
\Big( \mathcal{F}\left[\begin{array}{c}
u^{k-1}_{n-1} \\ 
\dot{u}^{k-1}_{n-1}\end{array} \right] - 
{\theta}\mathcal{C}\mathcal{R}
\left[ \begin{array}{c}
u^{k-1}_{n-1} \\
\dot{u}^{k-1}_{n-1} \end{array} \right] \Big) + 
\Big( \mathcal{F}\left[\begin{array}{c}
u^{k-1}_{n} \\ 
\dot{u}^{k-1}_{n}\end{array} \right] - 
\theta\mathcal{C}\mathcal{R}
\left[ \begin{array}{c}
u^{k-1}_{n} \\
\dot{u}^{k-1}_{n} \end{array} \right] \Big).
\end{align*}
It can be verified that the serial fine solution $[u(t_{n+1}),\dot{u}(t_{n+1})]$ also satisfies above expression when superscript $k,k-1$ are dropped in solution vector $[u^k_{\cdot},\dot{u}^k_{\cdot}]$. Then we have an expression for the difference of the solutions
\begin{align} \label{eq:errordiffexpand}
    \left[\begin{array}{c}
u_{n+1}^{k} - u(t_{n+1})\\
\dot{u}_{n+1}^{k} - \dot{u}(t_{n+1})
\end{array} \right] = & 
(\theta\mathcal{C}\mathcal{R})^{n-1}(\mathcal{F}-\theta\mathcal{C}\mathcal{R}) \left[\begin{array}{c}
u_{1}^{k-1} - u(t_1)\\
\dot{u}_{1}^{k-1} - \dot{u}(t_1)
\end{array} \right] + \dots
\\
&
(\theta\mathcal{C}\mathcal{R})(\mathcal{F}-\theta\mathcal{C}\mathcal{R}) \left[\begin{array}{c}
u_{n-1}^{k-1} - u(t_{n-1})\\
\dot{u}_{n-1}^{k-1} - \dot{u}(t_{n-1})
\end{array} \right] + \nonumber 
\\
& 
(\mathcal{F}-\theta\mathcal{C}\mathcal{R}) \left[\begin{array}{c}
u_{n}^{k-1} - u(t_{n})\\
\dot{u}_{n}^{k-1} - \dot{u}(t_{n})
\end{array} \right]. \nonumber
\end{align}
Recall the square root of energy error is defined as
$$
\mathcal{E}^{k}_{n+1}= 
\| \Lambda \left[\begin{array}{c}
u_{n+1}^{k} - u(t_{n+1})\\
\dot{u}_{n+1}^{k} - \dot{u}(t_{n+1})
\end{array} \right] \|_2.
$$
Using triangle inequality on $\mathcal{E}^k_{n+1}$ with equation (\ref{eq:errordiffexpand}) we obtain 
\begin{align*}
    \mathcal{E}^k_{n+1} & \leq \|\Lambda (\theta \mathcal{C} \mathcal{R})^{N-1} (\mathcal{F} - \theta \mathcal{C} \mathcal{R}) \left[\begin{array}{c}
    u^{k-1}_1 - u(t_1)\\
    \dot{u}^{k-1}_1 - \dot{u}(t_1)
    \end{array} \right] \|_2 \dots
    \\
    & +
    \|\Lambda (\theta \mathcal{C} \mathcal{R}) (\mathcal{F} - \theta \mathcal{C} \mathcal{R}) \left[\begin{array}{c}
    u^{k-1}_{n-1} - u(t_{n-1})\\
    \dot{u}^{k-1}_{n-1} - \dot{u}(t_{n-1})
    \end{array} \right]  \|_2 
    \\
    & + 
    \|\Lambda(\mathcal{F} - \theta \mathcal{C} \mathcal{R}) \left[\begin{array}{c}
    u^{k-1}_{n} - u(t_{n})\\
    \dot{u}^{k-1}_{n} - \dot{u}(t_{n})
    \end{array} \right] \|_2.
\end{align*}
Apply equation (\ref{eq:hypo1}) in Hypothesis~\ref{hypo:lipschitz} (i) to bound each term
\begin{align*}
    \mathcal{E}^k_{n+1} & \leq (1+\epsilon_{\theta})^{n-1}\|\Lambda (\mathcal{F} - \theta \mathcal{C} \mathcal{R}) \left[\begin{array}{c}
    u^{k-1}_1 - u(t_1)\\
    \dot{u}^{k-1}_1 - \dot{u}(t_1)
    \end{array} \right] \|_2 \dots
    \\
    & +
    (1+\epsilon_{\theta})\|\Lambda (\mathcal{F} - \theta \mathcal{C} \mathcal{R}) \left[\begin{array}{c}
    u^{k-1}_{n-1} - u(t_{n-1})\\
    \dot{u}^{k-1}_{n-1} - \dot{u}(t_{n-1})
    \end{array} \right]  \|_2 
    \\
    & + 
    \|\Lambda (\mathcal{F} - \theta \mathcal{C} \mathcal{R}) \left[\begin{array}{c}
    u^{k-1}_{n} - u(t_{n})\\
    \dot{u}^{k-1}_{n} - \dot{u}(t_{n})
    \end{array} \right] \|_2.
\end{align*}
Finally we use equation (\ref{eq:hypo2}) in Hypothesis~\ref{hypo:lipschitz} (ii) to obtain 
\begin{align*}
    \mathcal{E}^k_{n+1} & \leq (1+\epsilon_{\theta})^{N-1} \kappa \|\Lambda\left[\begin{array}{c}
    u^{k-1}_1 - u(t_1)\\
    \dot{u}^{k-1}_1 - \dot{u}(t_1)
    \end{array} \right] \|_2 \dots
    \\
    & +
    (1+\epsilon_{\theta}) \kappa \|\Lambda \left[\begin{array}{c}
    u^{k-1}_{n-1} - u(t_{n-1})\\
    \dot{u}^{k-1}_{n-1} - \dot{u}(t_{n-1})
    \end{array} \right]  \|_2 
    \\
    & + 
    \kappa \|\Lambda \left[\begin{array}{c}
    u^{k-1}_{n} - u(t_{n})\\
    \dot{u}^{k-1}_{n} - \dot{u}(t_{n})
    \end{array} \right] \|_2
    \\
    & =
    (1+\epsilon_{\theta})^{n-1} \kappa \mathcal{E}^{k-1}_{1}  \dots
    +
    (1+\epsilon_{\theta}) \kappa \mathcal{E}^{k-1}_{n-1}  
    + 
    \kappa \mathcal{E}^{k-1}_{n}
    \\
    & \leq
    \kappa \Big( (1+\epsilon_{\theta})^{n-1}  \dots
    +
    (1+\epsilon_{\theta}) +1 \Big)  
    \max_{j\leq n} \mathcal{E}^{k-1}_{j}
    \\
    & = \kappa \dfrac{(1+\epsilon_{\theta})^{n}-1}{\epsilon_{\theta}}
    \max_{j\leq n}~\mathcal{E}^{k-1}_{j}.
\end{align*}
Thus 
\begin{align*}
    \max_{j\leq N}~\mathcal{E}^{k}_{j} \leq \kappa \dfrac{(1+\epsilon_{\theta})^{N}-1}{\epsilon_{\theta}}
    \max_{j\leq N}~\mathcal{E}^{k-1}_{j}.
\end{align*}

By assumption $\kappa N < 1$ and $\epsilon_{\theta} N \ll 1$, the error goes to zero as $k$ approaches infinity.  
\end{proof}

We see that the convergence depends on the Lipschitz constant $\kappa$ in Hypothesis~\ref{hypo:lipschitz} (ii), which reflects the gap between the corrected coarse propagator to the fine propagator. This gap between propagators is quantified by the energy residual of the minimization \eqref{eq:procurstesform}. 

\section{Numerical Study of the New Algorithm}
In this section, we study the influence of different components of the proposed algorithm to the overall stability and accuracy.
From Section \ref{sec:ranktol} to Section \ref{sec:interpinfluence}, we consider the influence of (i) varying the low-rank approximation of the optimal phase correctors $\mathsf{\Omega}_*$, (ii) effect of the phase corrector and the parareal update,
(iii) different orders of approximation for the gradient $\nabla_h$ and interpolation operator $\mathcal{I}. $
Regarding to the last item, we will use the following interpolation methods, written as MATLAB functions, in this section:
\begin{itemize}
\item $\texttt{interpft}$: Fourier interpolation 
\item $\texttt{akima}$: cubic Hermite interpolation
\item $\texttt{pchip}$: cubic interpolation
\item $\texttt{linear}$: linear interpolation
\end{itemize}

From Section \ref{sec:ranktol} to Section \ref{sec:interpinfluence}, we shall consider the simplest one dimensional setting with $c\equiv 1$ for both coarse and fine propagator, and the initial data:
\begin{align*}
u(x,0)  &=\cos(10\pi x)\exp(-100x^2), ~~ x\in[-0.5,0.5]\\
u_t(x,0)&=0.
\end{align*}

For Section \ref{sec:randsubsample}, we consider random subsampling of the data matrices to exploit their observed low rank property. In this study, we consider a two dimensional problem with variable wave speed.
\newline
\newline
We will assume that the coarse grid nodes overlap with the fine grid nodes, and that the restriction operator $\mathcal{R}$ is just a point-wise evaluation on the coarse grid nodes.

The errors at final time $T_N=T$ are defined as square root of energy of difference on the fine grid
$$\sqrt{\frac{E([u^{k}_{N}-u(t_{N}),\dot{u}^{k}_{N}-\dot{u}(t_{N})])}{E([u(t_{N)},\dot{u}(t_{N})])}}.$$ And similarly the error can also be defined in $\ell^2$ of difference in displacement component
$$\frac{\| u^{k}_{N}-u(t_{N}) \|_2}{\|u(t_{N})\|_2}.$$
The reference solution $[u(t_N),\dot{u}(t_N)]$ are serially computed using the fine propagators.

\subsection{Rank tolerance of the phase corrector}\label{sec:ranktol}

In this example, we study the sensitivity of the algorithm to rank-truncation of the optimal phase corrector $\mathsf{\Omega_*}.$
We use the same spatial grid for both the coarse and the fine propagators in order to avoid error coming from interpolation/restriction. 
The fine propagator has an CFL number that is 20 times smaller than the coarse, and the coupling take place every 10 coarse steps. We sample several values for tolerance in Algorithm \ref{updatesvd} at $10^{-15},10^{-12},10^{-9},10^{-6},10^{-3}$. The parameters are tabulated below:
\bigskip
\begin{center}
\begin{tabular}{ c|c|c|c|c|c|c|c|c}
 $T$ & $\Delta t_{com}$ & $\Delta x$ & $\Delta t/\Delta x$ & $\Delta x/\delta x$ &  $\Delta t/\delta t$ & $\mathcal{I}$ & $\nabla_h$ & tol \\ 
 \hline
 $5$ & $0.05$ & $0.01$ & $0.5$ & $1$ & $20$ & $\texttt{interpft}$ & 2 order & \textcolor{blue}{$10^{\{-15,-12,-9,-6,-3\}}$} \\
\end{tabular}
\end{center}\bigskip

\begin{figure}
    \centering
    \includegraphics[width=0.45\linewidth]{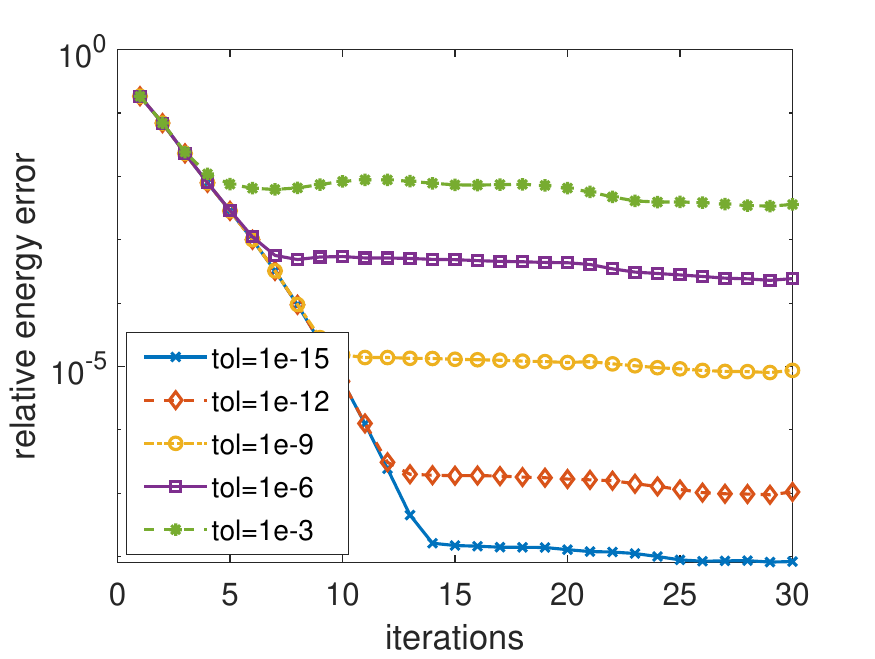}
    \includegraphics[width=0.45\linewidth]{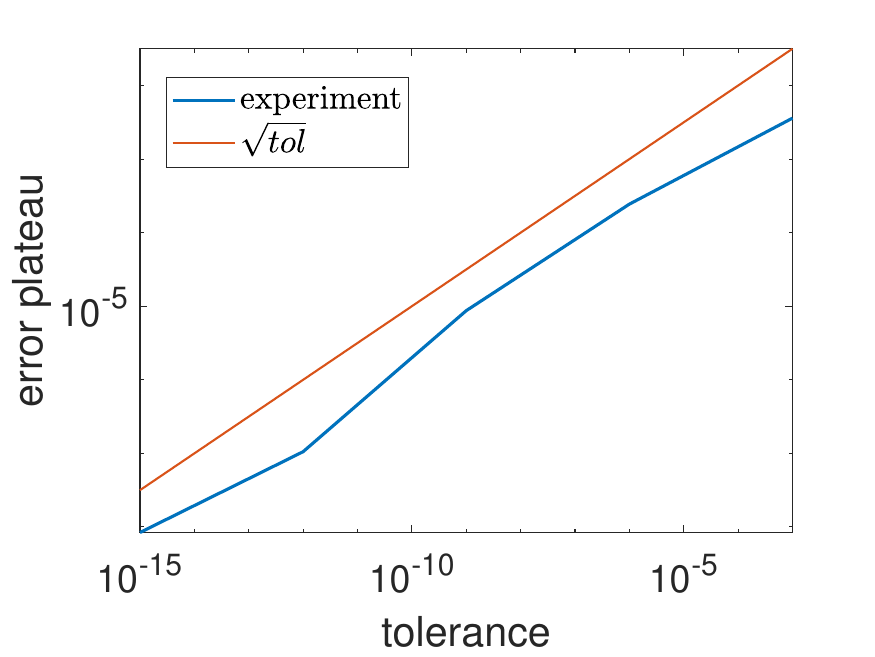}
    \caption{Dependence of error convergence on rank tolerance which is in Algorithm \ref{updatesvd}. Left: relative energy error as a function of iterations. Right: the stagnated error value as a function of tolerance.}
    \label{fig:ranktolerror}
\end{figure}

\begin{figure}
    \centering
    \includegraphics[width=0.5\linewidth]{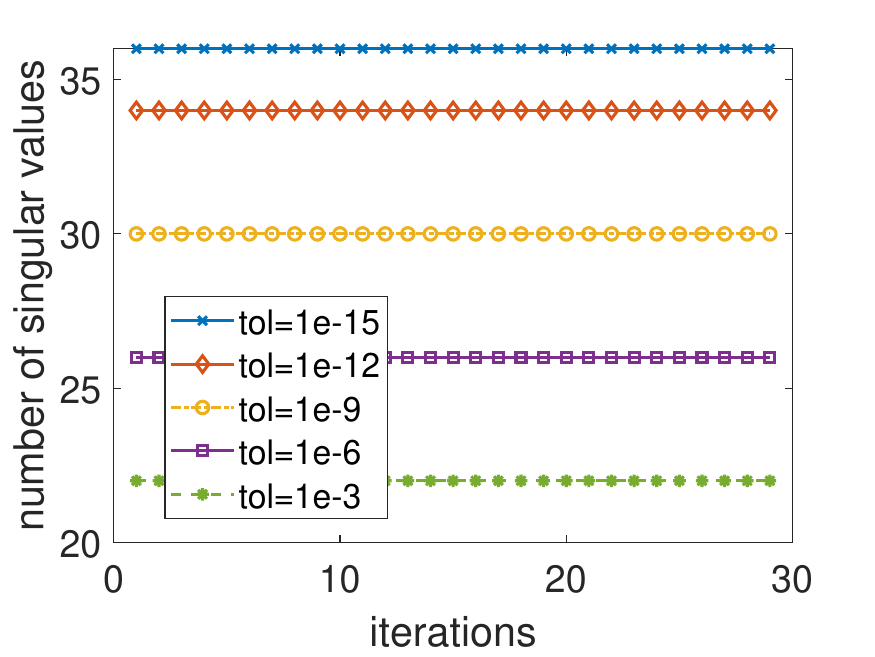}
    \caption{Number of singular values for different value of tolerance.}
    \label{fig:numranktol}
\end{figure}

Figure \ref{fig:ranktolerror} shows the relative \reviewertwo{energy} error along with the iterations as the tolerance in the truncation of $\mathsf{\Omega}_*$ is  varied. The errors decrease in the first few. The rate of decrease seem independent of the chosen tolerance values. As more iterations progress, the errors convergence eventually stagnate at certain values that strongly correlate to the chosen tolerance values. Particularly, the stagnated error values scale as the square root of the tolerance as shown on the right plot of Figure \ref{fig:ranktolerror}. This scaling can be explained by the fact that the tolerance 
corresponds to the truncation of $\mathsf{\Omega}_*$, which modifies the wave energy components, and we measure the square root of wave energy difference. Hence in general, the convergence rate of our method is expected to slow down after the error has passed $ 10^{-8}$ because the tolerance can only be as small as machine epsilon $10^{-16}$.
Figure \ref{fig:numranktol} shows the number of retained singular values for different values of tolerance. 

\subsection{The effect of phase correction (\textcolor{blue}{$\mathsf{\Omega}\equiv 1$})}
\begin{figure}
    \centering
    \includegraphics[width=0.45\linewidth]{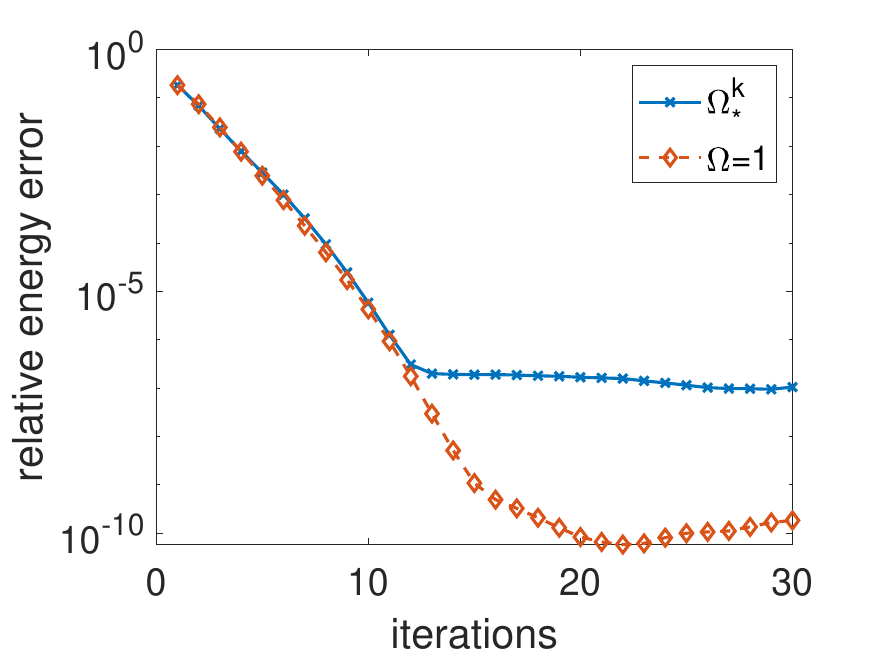}
    \includegraphics[width=0.45\linewidth]{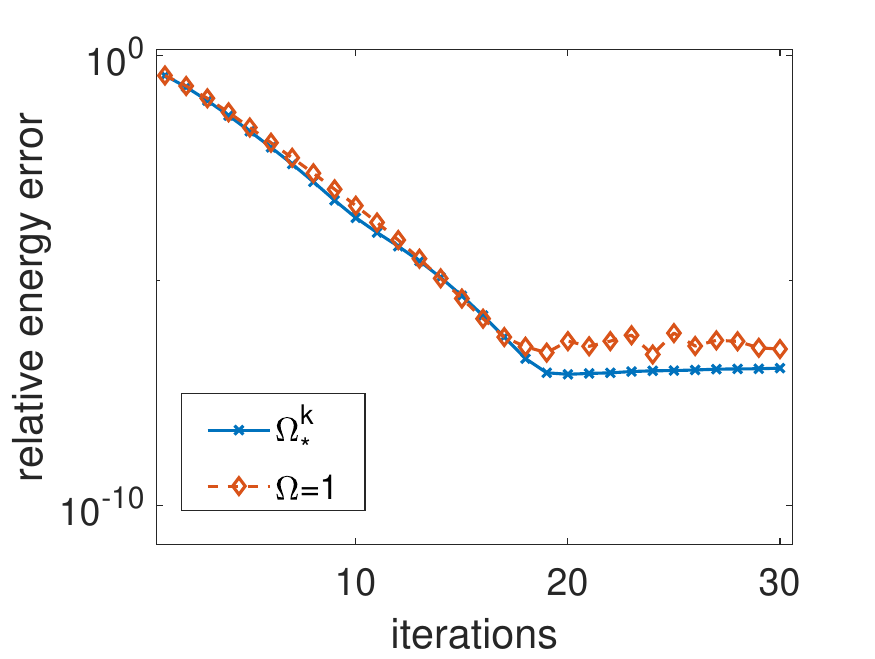}\\
    \includegraphics[width=0.45\linewidth]{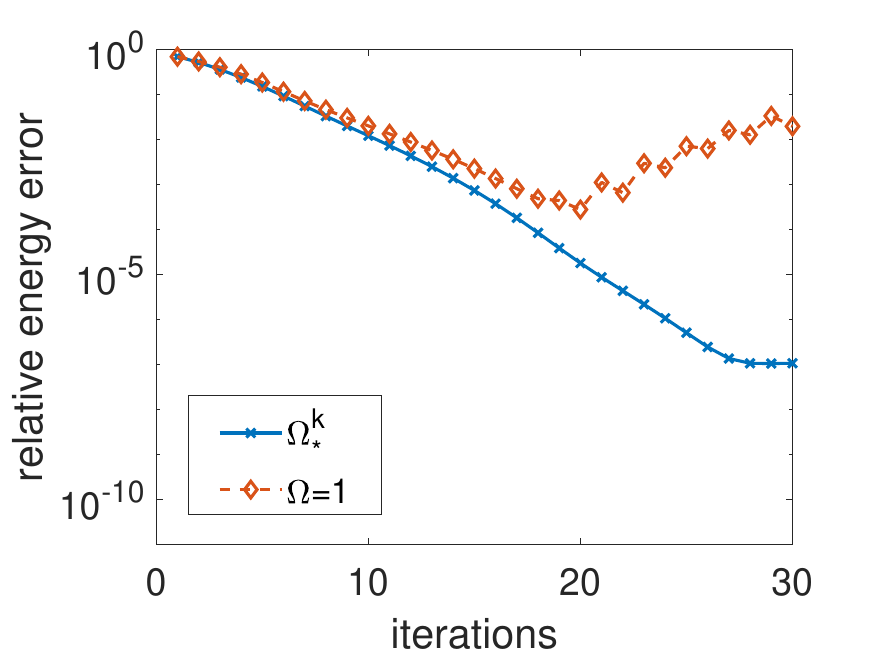}
    \includegraphics[width=0.45\linewidth]{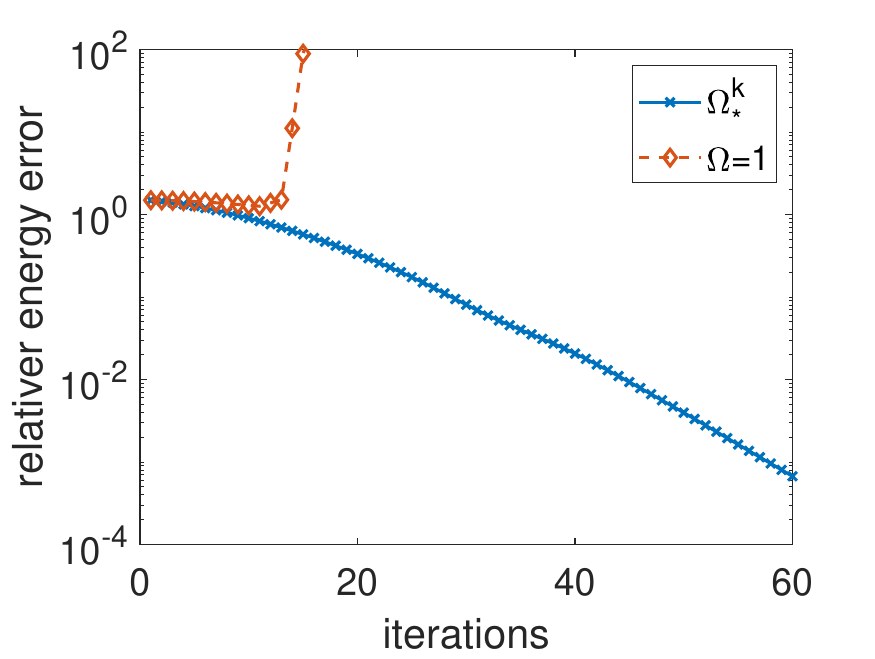}
    \caption{Comparison of error convergence at different terminal time. Top row, left: $T=2.5$, right: $T=5$. Bottom row, left: $T=10$, right: $T=50$. Applying optimized $\mathsf{\Omega}_*$ to solution is shown in cross solid curve. No optimization $\mathsf{\Omega}=1$, but fine and coarse solutions are coupled in wave energy components, is shown in diamond dashes. 
    }
    \label{fig:omega1}
\end{figure}

Assuming again that the coarse and fine propagators are on the same grid. Without the phase correction, i.e.  $\mathsf{\Omega}\equiv 1$, the proposed iteration takes the form 
$$
\left[\begin{array}{c}
u_{n}^{k}\\
\dot{u}_{n}^{k}
\end{array}\right] = 
\Lambda^\dag \Lambda \mathcal{C}
\left[\begin{array}{c}
u^{k}_{n-1} \\
\dot{u}^{k}_{n-1}
\end{array} \right] +
\mathcal{F}\left[\begin{array}{c}
u^{k-1}_{n-1} \\ 
\dot{u}^{k-1}_{n-1}\end{array} \right] - 
\Lambda^\dag \Lambda \mathcal{C} \left[ \begin{array}{c}
u^{k-1}_{n-1} \\
\dot{u}^{k-1}_{n-1} \end{array} \right].
$$
The above expression becomes the plain parareal method if the term $\Lambda^\dag \Lambda = 1$. But when the first wave component $\nabla_h u$ is approximated by some finite \reviewertwo{difference}, the term $\Lambda^\dag \Lambda \neq 1$ in general. In particular when $\nabla_h$ is approximated by the standard second order central \reviewertwo{difference}, 
i.e. $\nabla_h=D_{\Delta x}^0$, $\Lambda^\dag \Lambda$ corresponds to multiplication of 
$$ \dfrac{\sin{\xi \Delta x}}{\xi \Delta x} = \mathrm{sinc}(\xi \Delta x)$$ to the Fourier mode of the solutions. 
Since $|\mathrm{sinc}(\xi \Delta x) |\leq 1$,  $\Lambda^\dag \Lambda$  damps high frequency modes, and thus stabilizes parareal-like iterations. 

Nevertheless, for long time simulations, such high frequency damping may \reviewertwo{be insufficient} to stabilize the parareal-like iterations. To illustrate this, we take the same discretization as above but now consider four terminal times $ T=2.5 ,5 ,10, 50$:
\bigskip
\begin{center}
\begin{tabular}{ |c|c|c|c|c|c|c|c|c|c|}
\hline
 $T$ & $\Delta t_{com}$ & $\Delta x$ & $\Delta t/\Delta x$ & $\Delta x/\delta x$ &  $\Delta t / \delta t$ & $\mathcal{I}$ & $\nabla_h$ & tol\\ 
 \hline
  * & $0.05$ & $0.01$ & $0.5$ & $1$ & $20$ & $\texttt{interpft}$ & 2 order FD & $10^{-14}$ \\
  \hline
\end{tabular}
\end{center}\bigskip
$*:\textcolor{blue}{\{2.5,5,10,50\}}.$
Figure \ref{fig:omega1} presents a comparison of the errors computed with 
$\mathsf{\Omega}\equiv 1$ and with $\mathsf{\Omega}=\mathsf{\Omega}^k_*$, for different terminal times. 
For shorter time intervals, such as $T=2.5$, the two choices of $\mathsf{\Omega}$ yield similar convergence rates until after some iterations when the errors computed with $\mathsf{\Omega}_*$ plateau around a much larger value. For larger terminal times, $T=5,10,50$, the instability that comes with using $\mathsf{\Omega}\equiv 1$ becomes more and more apparent, while
the computations with $\mathsf{\Omega}=\mathsf{\Omega}^k_*$ remain stable.

\subsection{The effect of parareal-like corrections}
\begin{figure}
    \centering
    \includegraphics[width=0.45\linewidth]{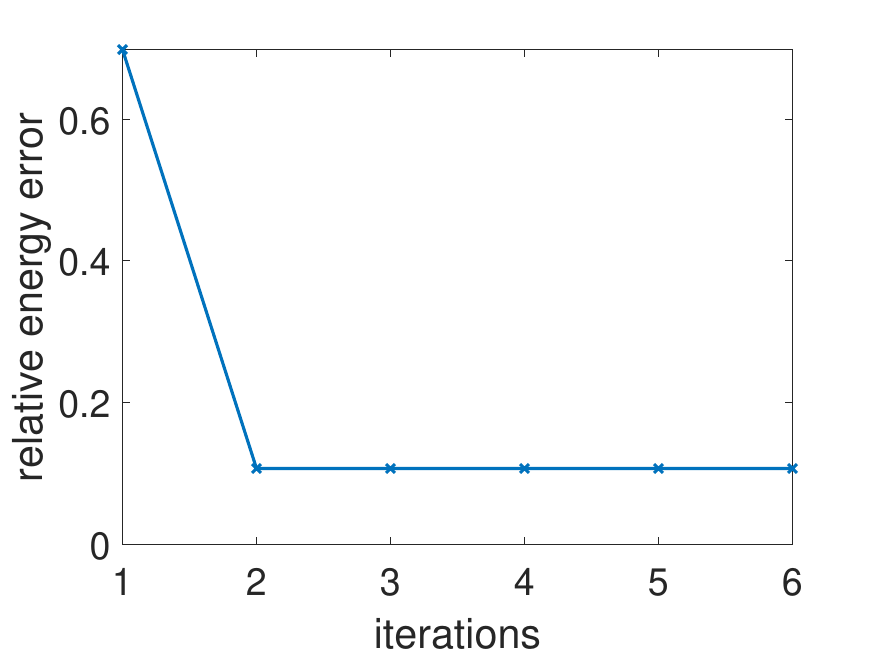} 
    \includegraphics[width=0.45\linewidth]{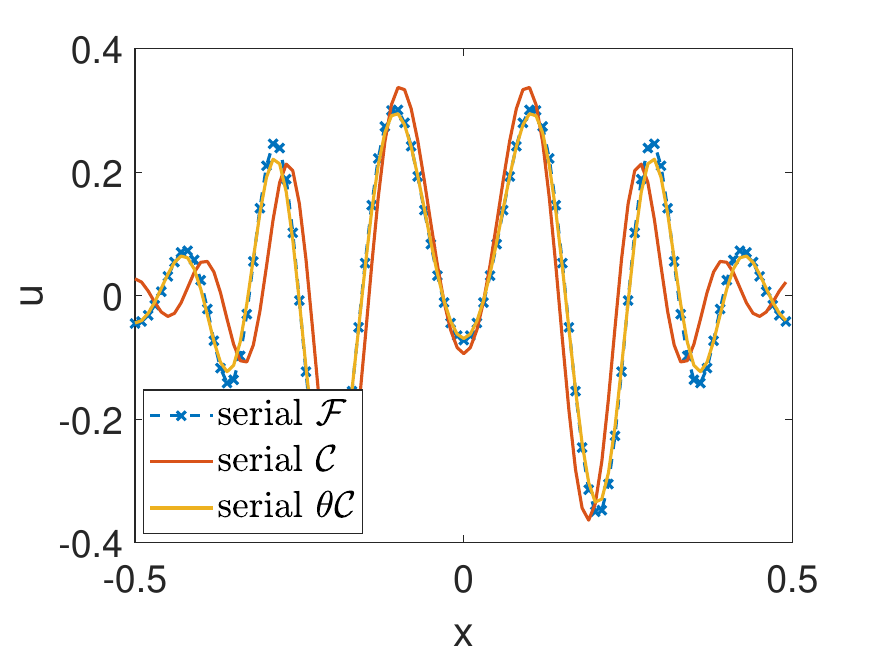} 
    \caption{The solution computed serially by the phase-corrected coarse propagator. Left: relative energy error of the phase-corrected coarse solution. Right: comparison with the serial fine and serial coarse solutions at $T=10$. }
    \label{fig:nopararealupdate}
\end{figure}

If the parareal-style additive correction is omitted, solution is propagated with just the phase corrected coarse propagator: 
\begin{align*}
\left[\begin{array}{c}
u_{n}^{k}\\
\dot{u}_{n}^{k}
\end{array}\right] = 
\theta^{k-1} \mathcal{C}
\left[\begin{array}{c}
\mathcal{R}u^{k}_{n-1} \\
\mathcal{R}\dot{u}^{k}_{n-1}
\end{array} \right].
\end{align*}

The simulation parameters are given as follow
\begin{center}
\begin{tabular}{|c|c|c|c|c|c|c|c|c|c|}
\hline
 $T$ & $\Delta t_{com}$ & $\Delta x$ & $\Delta t/\Delta x$ & $\Delta x/\delta x$ &  $\Delta t / \delta t$ & $\mathcal{I}$ & $\nabla_h$ & tol\\ 
 \hline
 $10$ & $0.05$ & $0.01$ & $0.5$ & $1$ & $20$ & $\texttt{interpft}$ & 4 order & $10^{-14}$ \\
 \hline
\end{tabular}
\end{center}\bigskip

We first point out that if $\mathcal{C}$ preserves the discrete wave energy, then the above scheme will also preserve it by construction of $\theta^k$.
Figure~\ref{fig:nopararealupdate} shows the errors comparing to the serial fine solution. At iteration $k=1$, the solution is serially computed with the coarse propagator $\mathcal{C}$. At iteration $k=1$, a phase corrector  $\theta^2$ is constructed based on the data computed in $k=1$. The solution at $k=2$ is serially computed with $\theta \mathcal{C}$. 
On the right subplot of Figure~\ref{fig:nopararealupdate}, we see that
the coarse solution now has the same phase as the fine solution, but has a slightly different amplitude.
For iteration after $k=3$, however, the error does not decrease further since the parareal-style additive correction has been omitted. 
Comparing to the examples with similar simulation parameters presented in the previous subsection, we see that the parareal-style correction
$$\mathcal{F}[u^{k-1}_{n-1},\dot{u}^{k-1}_{n-1}] - \theta^{k-1}\mathcal{C}[\mathcal{R}u^{k-1}_{n-1},\mathcal{R}\dot{u}^{k-1}_{n-1}]$$
is important, as it adds the missing amplitudes back to improve accuracy (when the solutions are properly aligned).

\subsection{Influence of interpolation and gradient approximation}\label{sec:interpinfluence}
\begin{figure}
    \centering
    \includegraphics[width=0.5\linewidth]{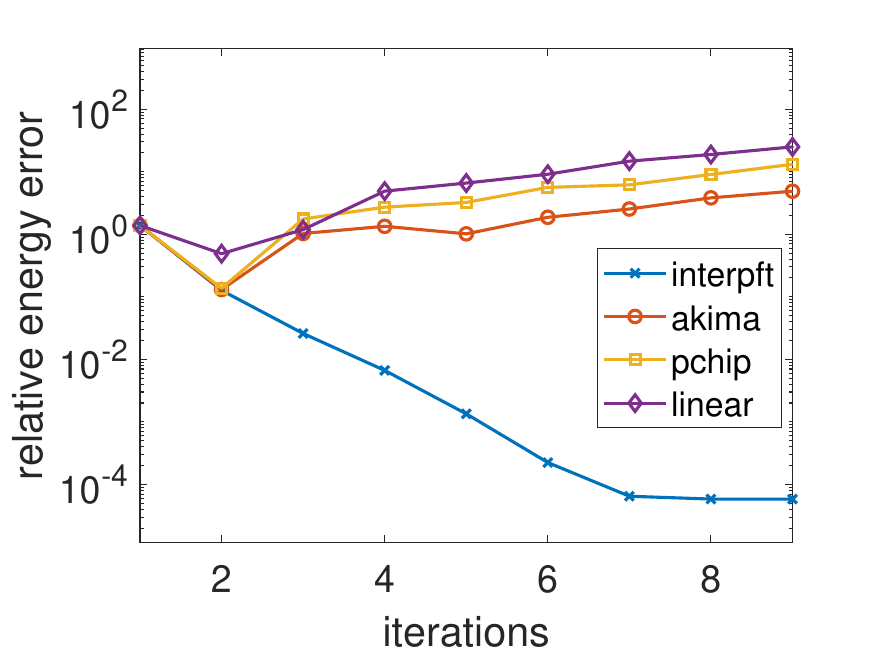}
    \caption{Error convergence of the proposed method using different interpolation schemes. 
    }
    \label{fig:interpmethods}
\end{figure}

So far in this section, we have only considered examples in which the coarse and fine propagators operate on the same spatial grid.  When these propagators are on two different grids, interpolation is needed to couple the solutions. In this subsection, we study the effect of interpolation. To illustrate this point, take coarse/fine grid ratio to be 2 and keep the discretization as before \bigskip
\begin{center}
\begin{tabular}{|c|c|c|c|c|c|c|c|c|c|}
\hline
 $T$ & $\Delta t_{com}$ & $\Delta x$ & $\Delta t/\Delta x$ & $\Delta x/\delta x$ &  $\Delta t / \delta t$ & $\mathcal{I}$ & $\nabla_h$ & tol\\ 
 \hline
 $10$ & $0.05$ & $0.01$ & $0.5$ & $10$ & $200$ & $*$ & 4 order & $10^{-14}$ \\
 \hline
\end{tabular}
\end{center}\bigskip
$*:\textcolor{blue}{\{\texttt{interpfft},\texttt{akima},\texttt{pchip},\texttt{linear}\}}.$ Input wave speed for coarse propagator is $c=1$ as well.
Figure \ref{fig:interpmethods} shows the error convergence with different methods for grid interpolation. We observe \reviewerone{particularly for this example} that the spectral interpolation $\texttt{interpft}$ performs \reviewerone{better than the lower order} methods because it resolves the initial wave form much better.

\begin{figure}
    \centering
    \includegraphics[width=0.5\linewidth]{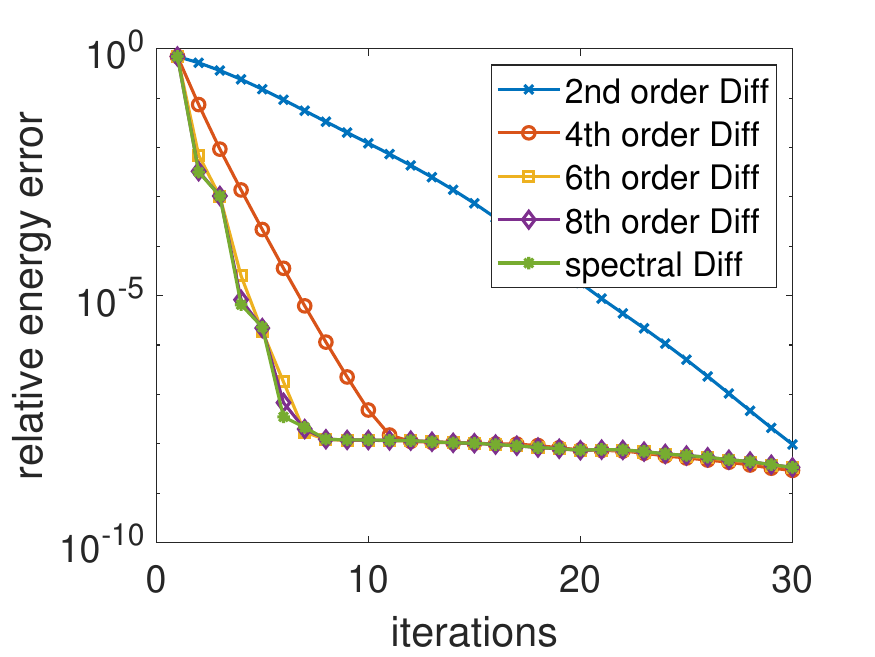}
    \caption{Relative energy errors at $T=10$, computed with different order of approximations to $\nabla_h U$. 
    }
    \label{fig:difforder}
\end{figure}
We also study the influence of the accuracy in approximating the gradient of the wavefield in forming the data matrices. We observe from the following examples that higher order approximations of gradient estimation accelerates  convergence rate of the proposed method. 
The parameters used in the simulations are tabulated below:
\begin{center}
\begin{tabular}{ |c|c|c|c|c|c|c|c|c|c|}
\hline
 $T$ & $\Delta t_{com}$ & $\Delta x$ & $\Delta t/\Delta x$ & $\Delta x/\delta x$ &  $\Delta t / \delta t$ & $\mathcal{I}$ & $\nabla_h$ & tol\\ 
 \hline
 $10$ & $0.05$ & $0.01$ & $0.5$ & $1$ & $20$ & \texttt{interpft} & $*$ order & $10^{-14}$\\
 \hline
\end{tabular}
\end{center}\bigskip
$*:\textcolor{blue}{\{2,4,6,8,spectral\}}.$
To isolate other factors that can also influence the convergence rate, the table below shows the relative residual in Sec \ref{sec:solnoptim} averaged over all iterations, denoted as $\Bigg \langle \|\mathsf{F} - \mathsf{\Omega}^k_*\mathsf{G} \|_{F} \Big/ \| \mathsf{F}\|_{F} \Bigg \rangle_{k}$. We see that the residual does not change while we increase the order of finite difference. 
In the last column, the errors in reconstruction of $U$ from the its approximated gradient is provided. To be specific, we denote operation in equation \eqref{eq:gradient2coordinate} as $Y:\nabla_h v \mapsto v$\reviewertwo{.}
\begin{center}
\begin{tabular}{ |c|c|c|}
\hline
 approx. order of $\nabla_h$ & $\Bigg \langle \dfrac{\|\mathsf{F} - \mathsf{\Omega}^k_*\mathsf{G} \|_{F}}{\| \mathsf{F}\|_{F}} \Bigg \rangle_{k}$ & $\max_{i,k} \dfrac{\| Y \nabla_h U^{k}_{i} - U^{k}_{i}\|_{2}}{\|U^{k}_{i}\|_2} $ \\ 
 \hline
 2 & $3.8634\cdot 10^{-3}$ & $ 2.9435\cdot 10^{-2} $ \\
 4 & $3.8101\cdot 10^{-3}$ & $ 1.4023\cdot 10^{-3} $ \\
 6 & $3.8079\cdot 10^{-3}$ & $ 8.2100\cdot 10^{-5} $ \\
 8 & $3.8077\cdot 10^{-3}$ & $ 5.8569\cdot 10^{-6} $ \\
 spectral & $3.8077\cdot 10^{-3}$ & $ 1.7706\cdot 10^{-12} $\\
 \hline
\end{tabular}
\end{center}
\bigskip
 Figure~\ref{fig:difforder} shows the convergence of errors for
 different central differencing and Fourier approximations for $\nabla_h$. 
 The ones with second order approximation has the slowest convergence rate, while those using sixth order or higher converge faster. 

\subsection{Random subsample of the data matrices} \label{sec:randsubsample}

We point out here that the cost of the stabilization can be further reduced by certain randomized algorithms \cite{tropp2019streaming,martinsson2017householder}. These randomized algorithms exploit the observed low-rank nature of our data matrices. Another approach is directly subsampling the data matrices. To illustrate the low-rank property, consider a plane wave in a 2D wave guide 
$$c(x,y) = 1-0.3\cos(2\pi x), ~~ -0.5\leq x \leq 0.5, -0.5\leq y \leq 0.5, $$ with the following discretization
    \bigskip
    \begin{center}
    \begin{tabular}{|c|c|c|c|c|c|}
    \hline
    $T$ & $\Delta t_{com}$ & $\Delta t$ & $\delta t$ & $N_{\Delta x}$ & $N_{\delta x}$  \\
    \hline
    $4$ & $0.02$ & $8\cdot 10^{-4}$ & $8\cdot 10^{-5}$ & $100\times 100$ & $200\times 200$ \\
    \hline
    \end{tabular}
    \end{center}
    \bigskip
    After each parallel computation of coarse and fine data $\mathsf{G},\mathsf{F}$, we plot the normalized strength of singular values of the correlation matrix $\mathsf{M}=\mathsf{F}\mathsf{G}^T$ for a few iterations in Figure \ref{fig:svM-randslice}.

\begin{figure}
    \centering
    \includegraphics[width=0.45\linewidth]{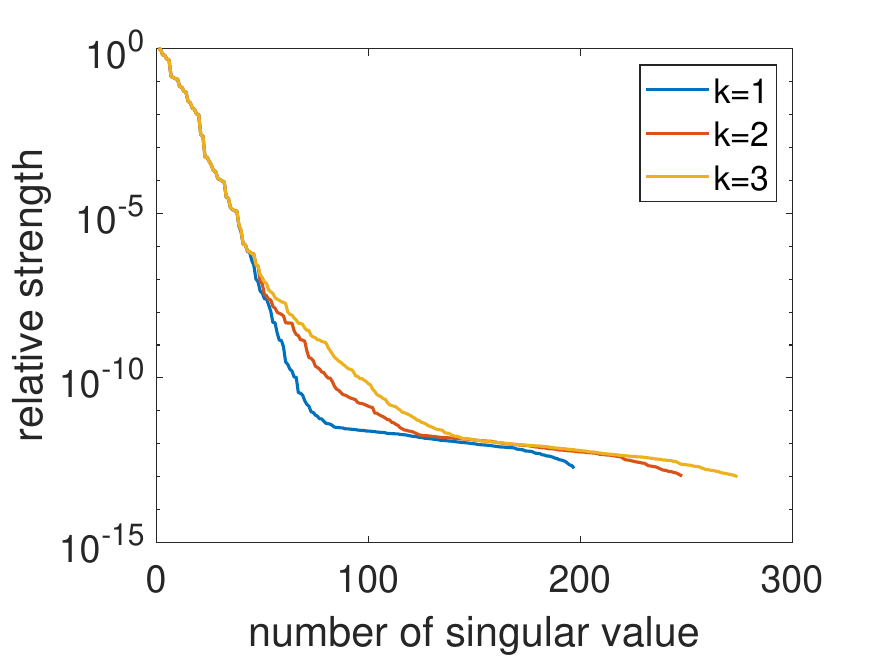}
    \includegraphics[width=0.45\linewidth]{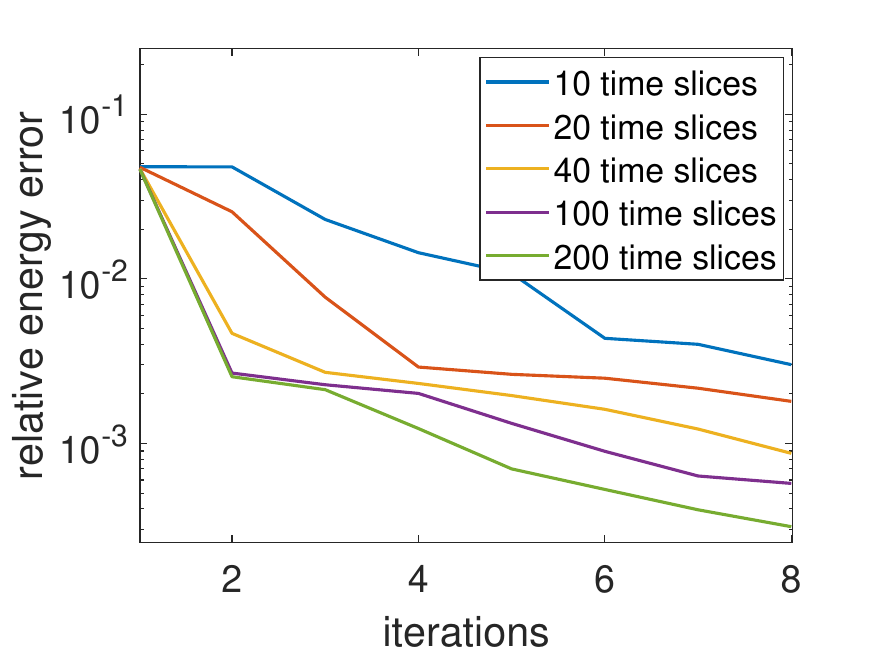}
    \caption{Low rank property of the data matrices in Section \ref{sec:randsubsample}. Left: normalized singular values of the correlation matrix $\mathsf{M}$. Right: relative energy errors at final time using random subsample of the data matrices.}
    \label{fig:svM-randslice}
\end{figure}

The normalized strength of the singular values drops exponentially in this particular example. A quick and simple strategy to exploit this low-rank property is to randomly sample time slices in matrices $\mathsf{F}$ and $\mathsf{G}$. By reducing the sample size, the data matrices becomes thinner so that QR factorization is faster. We compared the convergence of different sample sizes in Figure \ref{fig:svM-randslice}.

\section{Numerical Examples}
In this section,  we shall consider one and two dimension examples, including an example that involve a large scale wave speed model commonly used in the seismic migration community. When the spatial grid of coarse and fine are different, wave speed on the coarse grid is point wise evaluation of the given wave speed.

\subsection{One dimensional examples} \label{sec:oneDexample}
\begin{figure}
    \centering
    \includegraphics[width=0.45\linewidth]{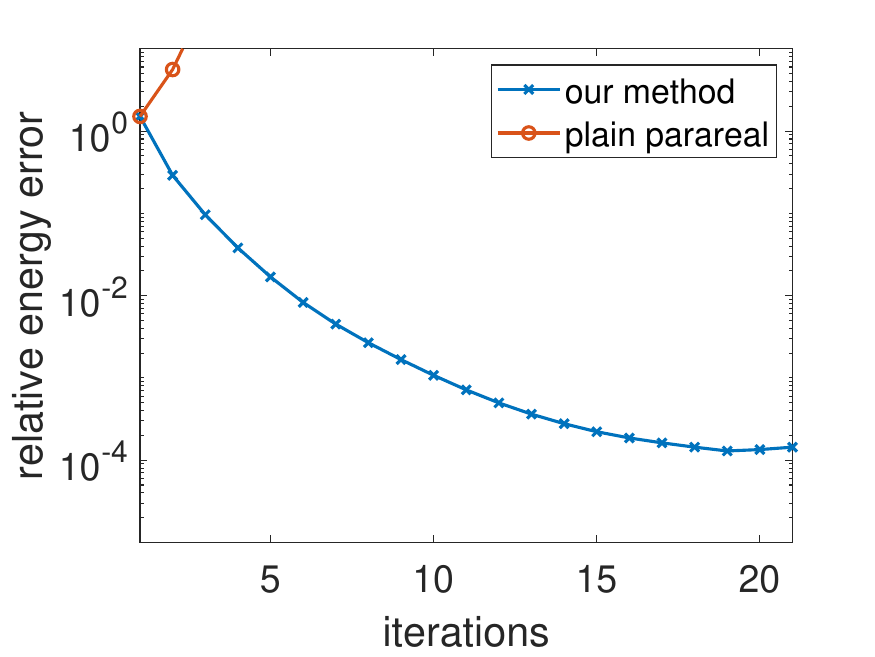}
    \includegraphics[width=0.45\linewidth]{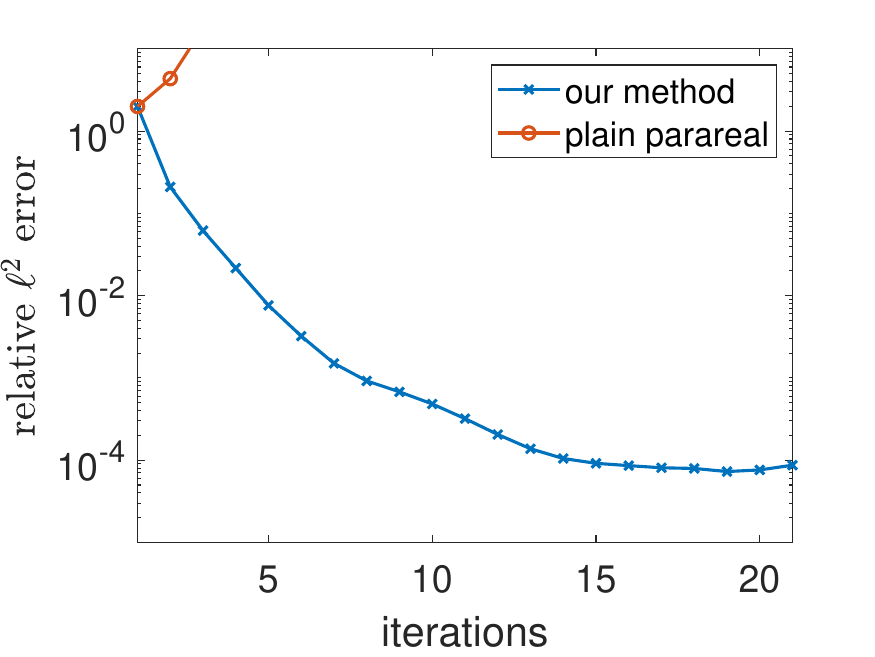}
    \caption{Relative error of the solutions computed in numerical example Sec.~\ref{sec:oneDexample}. Left: the energy error, right: the $\ell^2$ error. Our method, shown in cross solid line, generalizes beyond constant wave speed while the plain parareal method, shown in circle dash, diverges right away.}
    \label{fig:oneDvar}
\end{figure}

Consider a medium with the wave speed
$$
c(x)=1+0.25\cos(4\pi x),
$$
and the initial wavefield in $[-0.5,0.5]$
\begin{align*} 
 u(x,0) &=\cos(10\pi x)\exp(-100x^2),\\
 u_t(x,0)&=0.
 \end{align*}
 We present a numerical simulation using the parameters listed below:\bigskip

\begin{center}
\begin{tabular}{|c|c|c|c|c|c|c|c|c|c|}
\hline
 $T$ & $\Delta t_{com}$ & $\Delta x$ & $\Delta t/\Delta x$ & $\Delta x/\delta x$ &  $\Delta t / \delta t$ & $\mathcal{I}$ & $\nabla_h$ & tol\\ 
 \hline
 $10$ & $0.05$ & $0.01$ & $0.5$ & $10$ & $100$ & \texttt{interpft} & 4 order & $10^{-14}$\\
 \hline
\end{tabular}
\end{center}\bigskip
The fine propagator operate on a spatial grid which is 10 times finer than the coarse grid, and uses a CFL which is 10 times smaller. Figure~\ref{fig:oneDvar} shows convergence of the proposed method comparing to the plain parareal. 
Because fine and coarse solution in a variable medium may differ a lot, the plain parareal method becomes even more unstable.

\subsection{Two dimensional cases}
We apply the proposed method to three types of media: one with a smoothly varying wave speed (wave guide), one containing a piece-wise constant wave speed  (inclusion), and a more complicated wave speed profile which is often used in exploration seismology as a standard case study (Marmousi). 

\subsubsection{Waveguide} \label{sec:waveguide}
\begin{figure}
    \centering
    \includegraphics[width=0.45\linewidth]{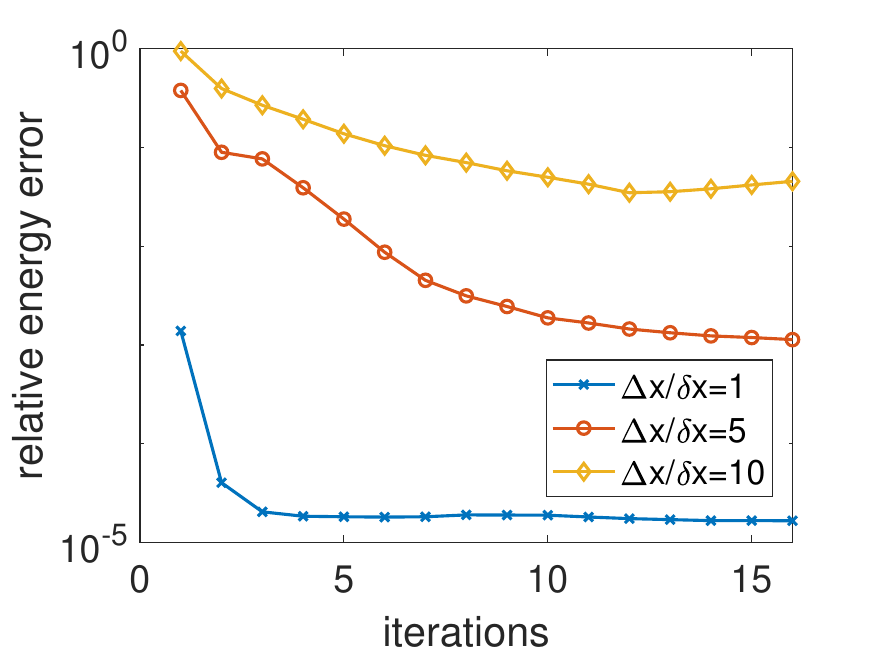}
    \includegraphics[width=0.45\linewidth]{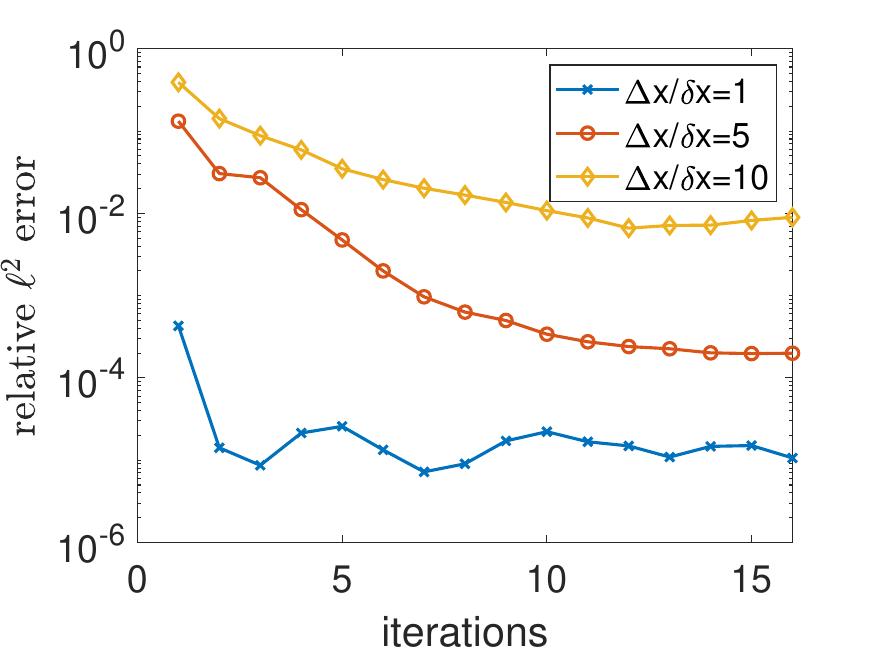}
    \caption{Relative error of parareal iterated solutions for the waveguide example in Sec \ref{sec:waveguide}. Left: the energy error, right: the $\ell^2$ error. }
    \label{fig:twoDwaveguide}
\end{figure}

We consider a wave guide in $xy$-plane $[-1,1]\times[-0.5,0.5]$ with the wave speed 
$$c(x,y)=1-0.3\cos(2\pi y).$$
The initial data is a plane wave traveling left to right along the $x$-axis:
\begin{align*}
u(x,y;0) &=\exp(-50(x+0.5)^{2}),\\
u_t(x,y;0) &=100\exp (-50(x+0.5)^{2}).
\end{align*}
The parameters used in the simulation are set as follow
\begin{center}
\begin{tabular}{|c|c|c|c|c|c|c|c|c|c|}
 \hline
 $T$ & $\Delta t_{com}$ & $\Delta x$ & $\Delta t/\Delta x$ & $\Delta x/\delta x$ &  $\lambda_{\Delta} / \lambda_{\delta}$ & $\mathcal{I}$ & $\nabla_h$ & tol\\ 
 \hline
 $5$ & $0.05$ & $0.005$ & $1/4$ & $\{1,5,10\}$ & $5$ & \texttt{interpft} & 4 order & $10^{-13}$\\
 \hline
\end{tabular}
\end{center}\bigskip

Figure \ref{fig:twoDwaveguide} shows error of the solution with different coarse fine grid ratio.

\subsubsection{Inclusion}\label{sec:inclusion-ex}
\begin{figure}
    \centering
    \includegraphics[width=0.45\linewidth]{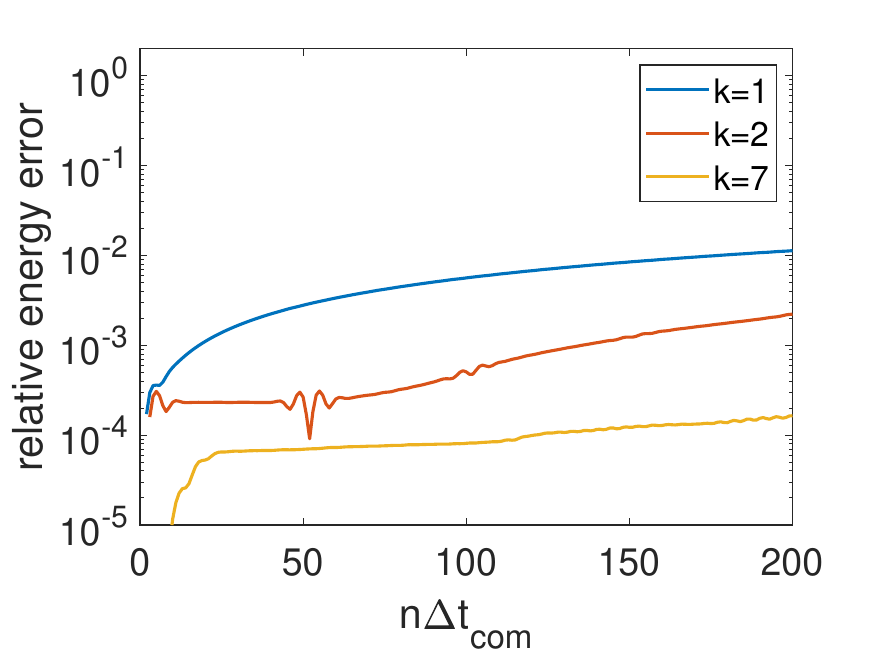}
    \includegraphics[width=0.45\linewidth]{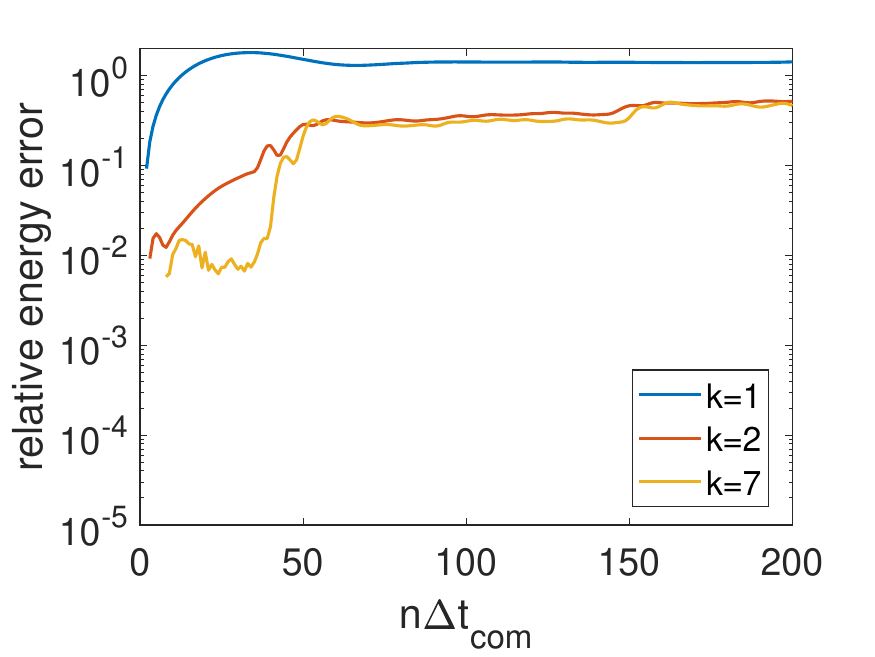}
    \caption{Relative error over time for the inclusion example described in Sec~\ref{sec:inclusion-ex}. The initial plane wave ``hits" the small inclusion at around $T=n\Delta t_{{com}}$. Left: the energy error for $\Delta x/\delta x=1$, right: the energy error for $\Delta x/\delta x=5$. }
    \label{fig:twoDincl_error}
\end{figure}

\begin{figure}
    \centering
    \includegraphics[width=0.45\linewidth]{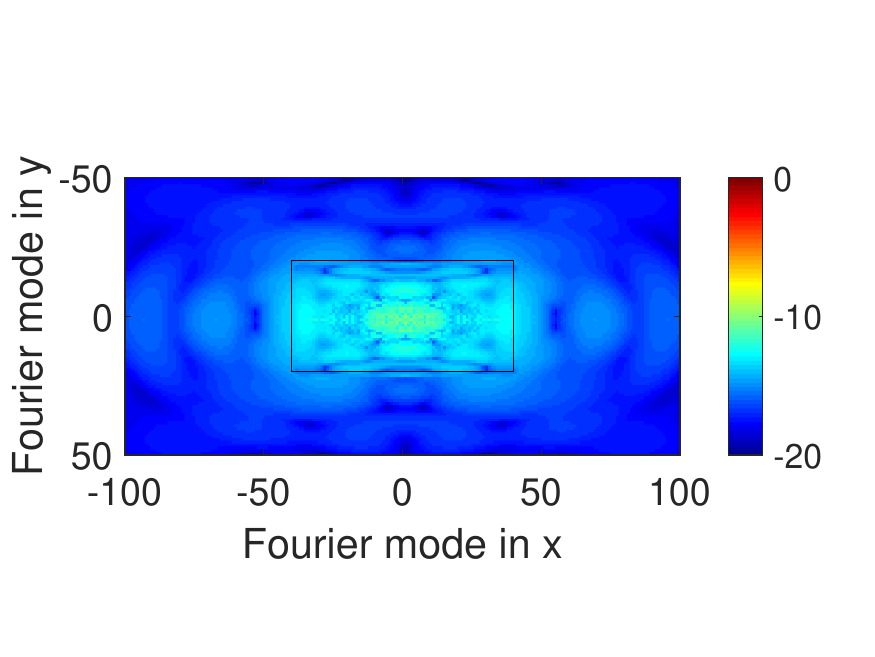}
    \includegraphics[width=0.45\linewidth]{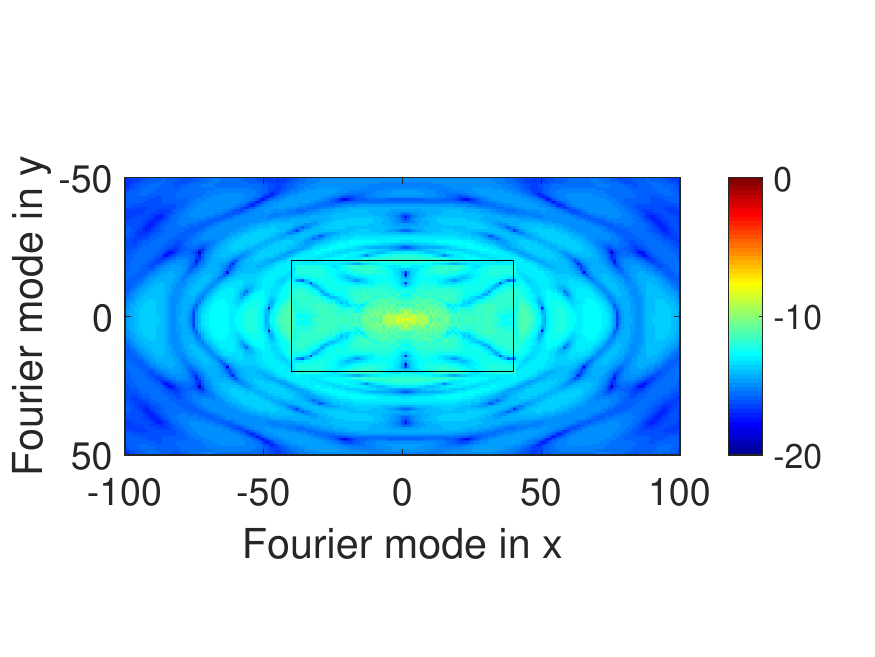} \\ 
    \includegraphics[width=0.45\linewidth]{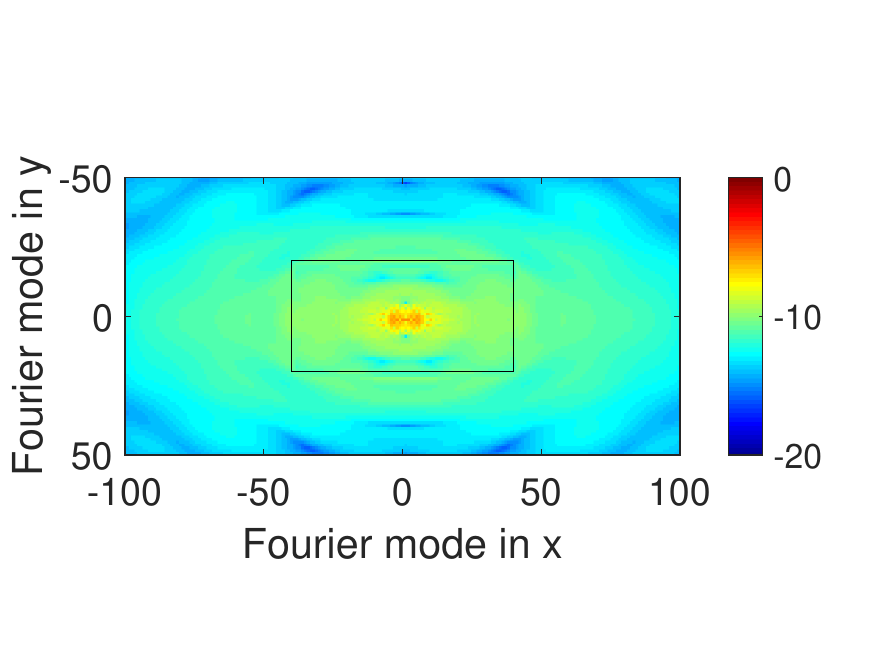}
    \includegraphics[width=0.45\linewidth]{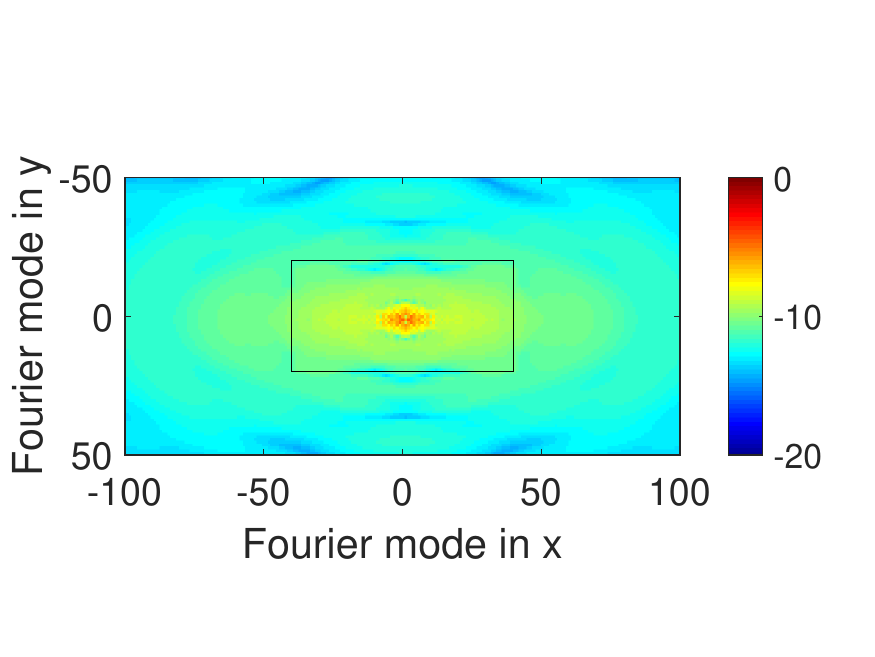}
    \caption{Relative density error of solution of inclusion example for $\Delta x/\delta x=5$ in Fourier modes $\dfrac{|\texttt{fft2}\{u^k_N(\xi_j) - u(\xi_j,t_N)\}|}{\sum_j |\texttt{fft2} \{u(\xi_j,t_N)\} |} $. The rectangular box indicates the Fourier domain of the coarse spatial grid. Top row, left: error at $n=15$, right: error at $n=40$. Bottom row, left: error at $n=140$, right: error at $n=200$.}
    \label{fig:twoDincl_Ferrorfield}
\end{figure}

In this example, we consider  the two dimensional domain in the $xy$-plane $[-1,1]\times[-0.5,0.5]$ where a plane wave encounters an inclusion of radius $\sqrt{0.002}$ centered at $[0.5,0.1]$, modeled by the wave speed
$$c(x,y)=1-0.9~\chi_{((x-0.5)^2+(y+0.1)^2)<0.002}.$$ We used the initial data traveling from left to right
\begin{align*}
    u(x,y;0) &=\cos(4\pi(x+0.5))\exp(-50(x+0.5)^2),\\
    u_t(x,y;0) &=\Big(-4\pi\sin(4\pi(x+0.5))+100(x+0.5)\cos(4\pi(x+0.5))\Big) \exp(-50(x+0.5)^2),
\end{align*}
and discretization parameters
\begin{center}
\begin{tabular}{ |c|c|c|c|c|c|c|c|c|c|}
\hline
 $T$ & $\Delta t_{com}$ & $\Delta x$ & $\Delta t/\Delta x$ & $\Delta x/\delta x$ &  $\lambda_{\Delta} / \lambda_{\delta}$ & $\mathcal{I}$ & $\nabla_h$ & tol \\ 
 \hline
 $4$ & $0.02$ & $0.005$ & $1/2$ & $\{1,5\}$ & $5$ & \texttt{interpft} & 4 order & $10^{-13}$\\
 \hline
\end{tabular}
\end{center}\bigskip

When the coarse grid is the same as fine grid, the iterations converge to the serial fine solution for the whole time interval (shown in left subplot of Figure~\ref{fig:twoDincl_error}). On the other hand, when coarse/fine grid ratio is $5$, the right subplot of Figure~\ref{fig:twoDincl_error} shows that the error escalates quickly at $n=50$ (or $t=1$), when the initial plane wave hits the inclusion for the first time, and again at $n=150$ (or $t=2$), as some parts of the initial plane wave wraps around the domain the interact with the inclusion again. The error does not decreasing for later iterations.  

Figure \ref{fig:twoDincl_Ferrorfield} shows the relative density error in the Fourier modes of the computed solution at different times. For the short time range $n=15$ (before the wave energy is scattered by the inclusion), most of the error concentrates at low frequencies which the coarse grid is able to resolve. Once the wave touches the inclusion at $n=40$ and thereafter $n=140,n=200$, the errors in the higher frequencies becomes significant. These scattered higher frequency wave is not resolved by the coarse grid and cannot be corrected by the proposed method.

\subsubsection{Marmousi experiment} \label{sec:Marm-ex}
\begin{figure}
    \centering
    \includegraphics[width=0.5\linewidth]{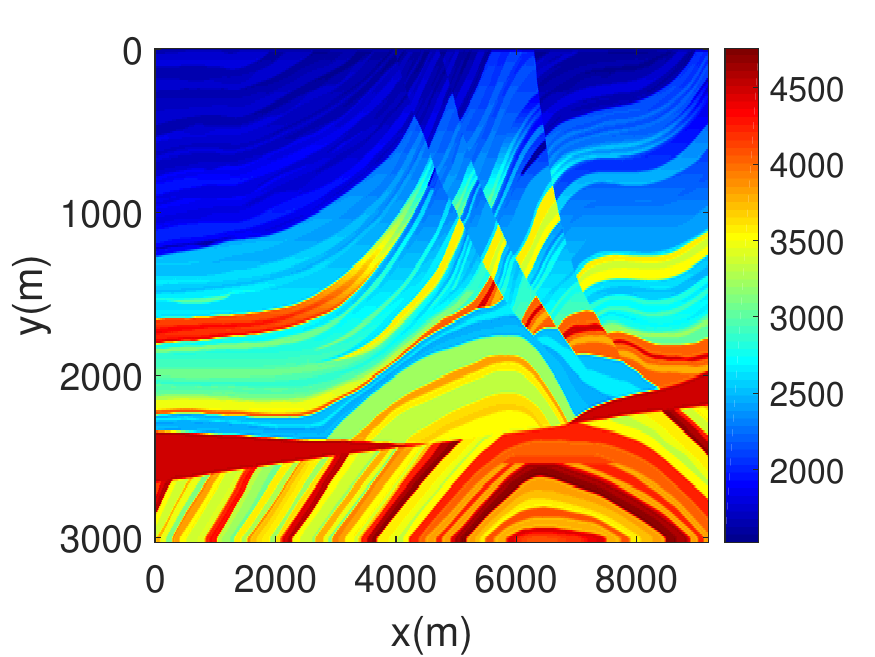}
    \caption{Marmousi wave speed model. Domain size is $3022 (m)\times 9192 (m)$ ($m$ denotes 'meter') and unit of wave speed is in meter per second.}
    \label{fig:twoDmarm_model}
\end{figure}

\begin{figure}
    \centering
    \includegraphics[width=0.45\linewidth]{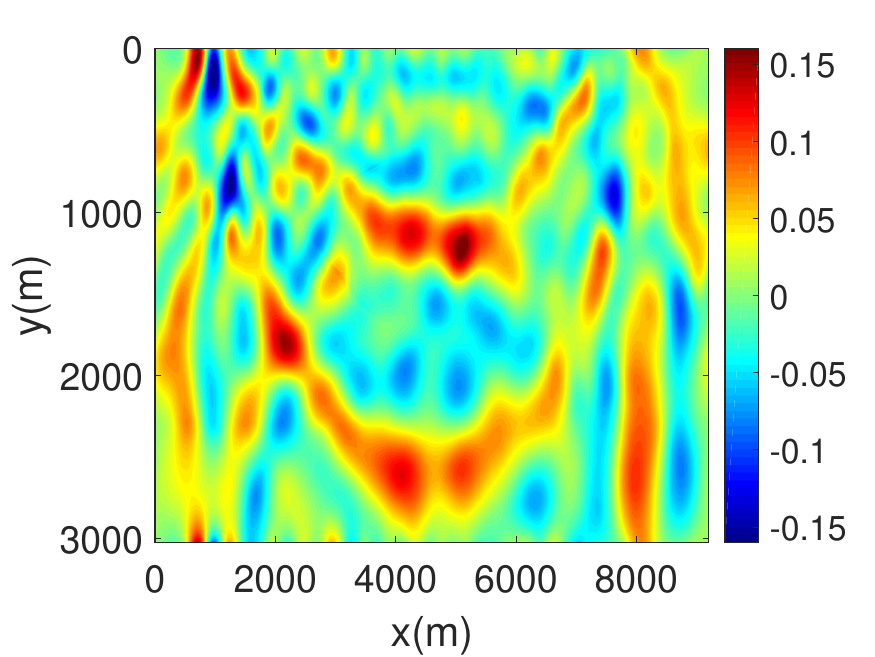}
    \includegraphics[width=0.45\linewidth]{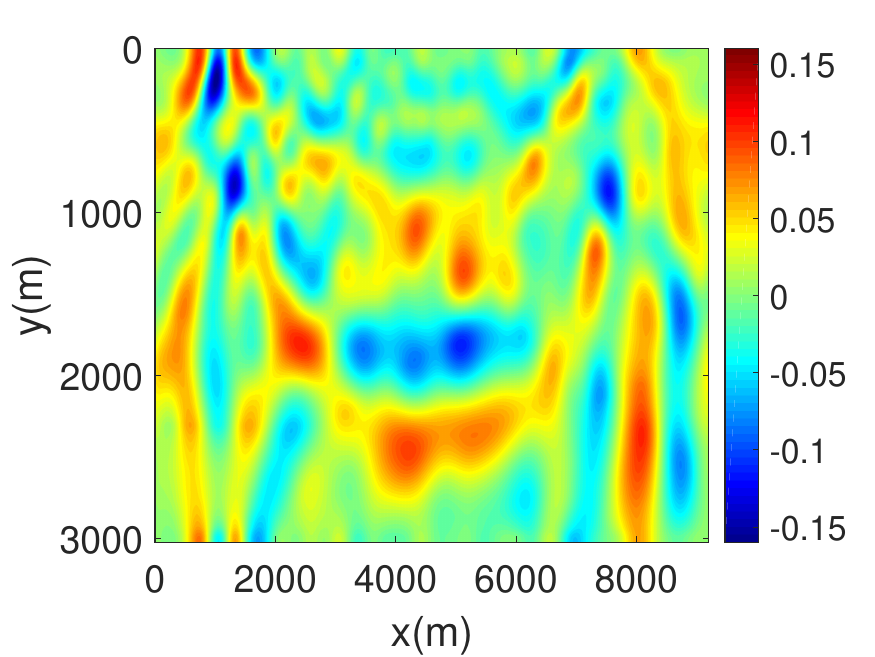} \\
    \includegraphics[width=0.45\linewidth]{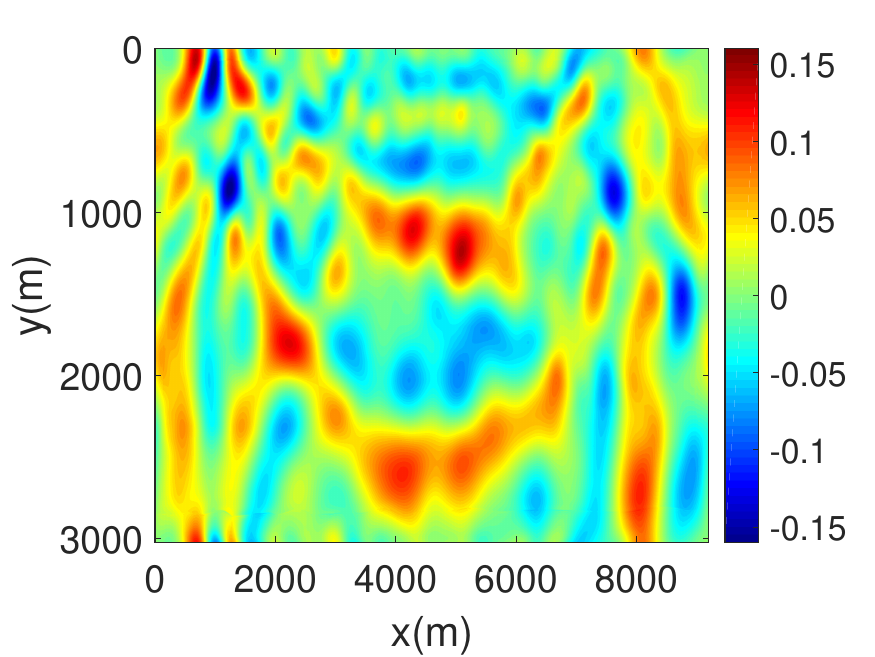}
    \includegraphics[width=0.45\linewidth]{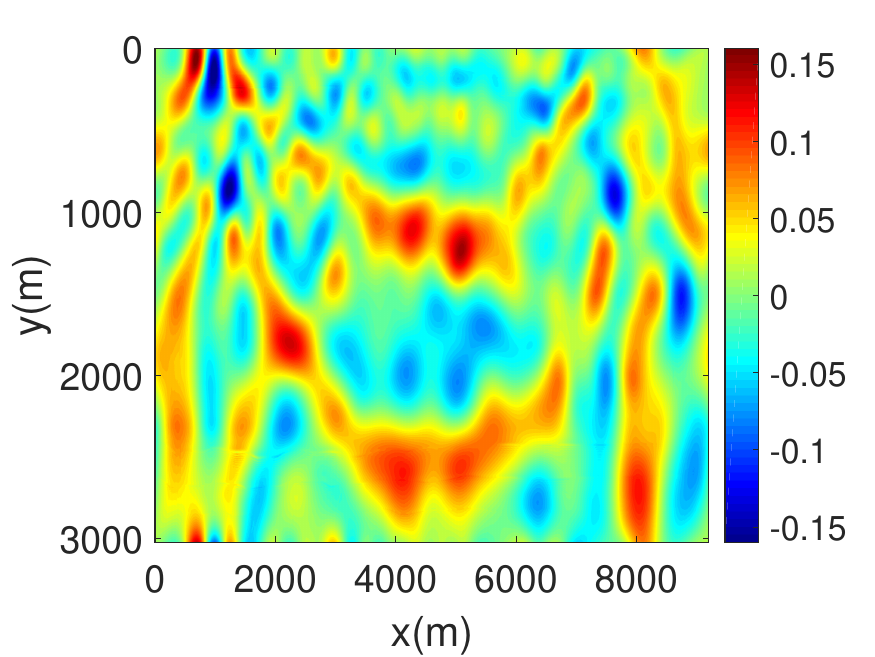}
    \caption{Solution at $T=2(s)$ of Marmousi example in Sec~\ref{sec:Marm-ex}. Top row, left: serial fine solution, right: serial coarse solution. Bottom row, left: $k=2$ iterated solution, right: $k=4$ iterated solution. }
    \label{fig:twoDmarm_solution}
\end{figure}

\begin{figure}
    \centering
    \includegraphics[width=0.45\linewidth]{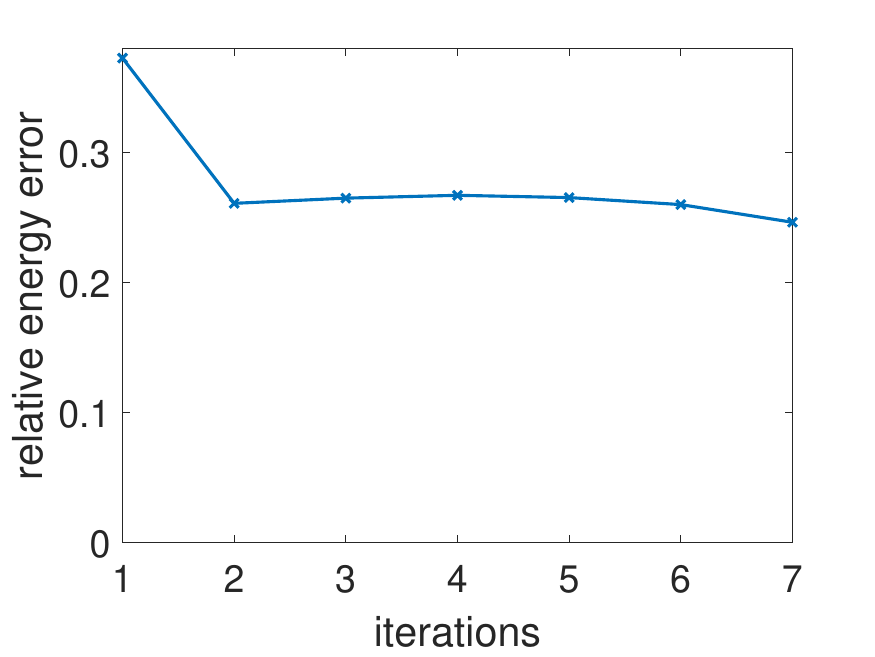}
    \includegraphics[width=0.45\linewidth]{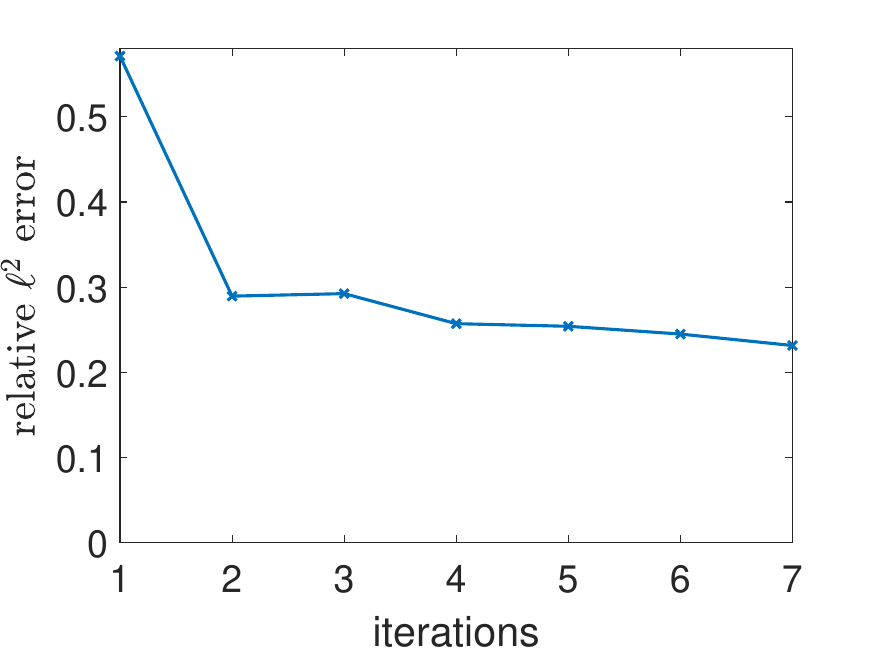}
    \caption{Relative errors by the proposed method applied to simulate wave propagation in the Marmousi model. Left: the energy error, right: the $\ell^2$ error.}
    \label{fig:twoDmarm_error}
\end{figure}

\begin{figure}
    \centering
    \includegraphics[width=0.45\linewidth]{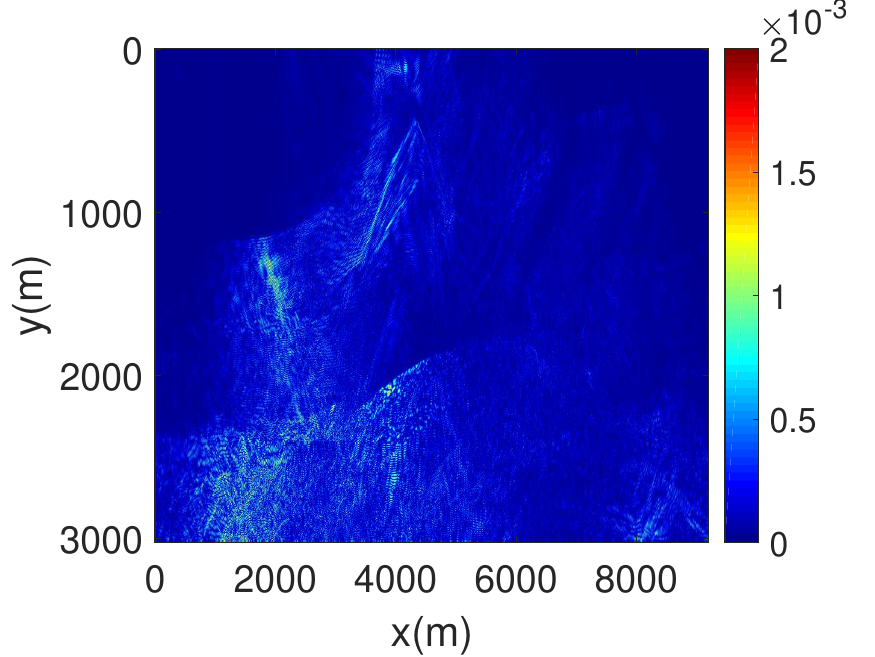}
    \includegraphics[width=0.45\linewidth]{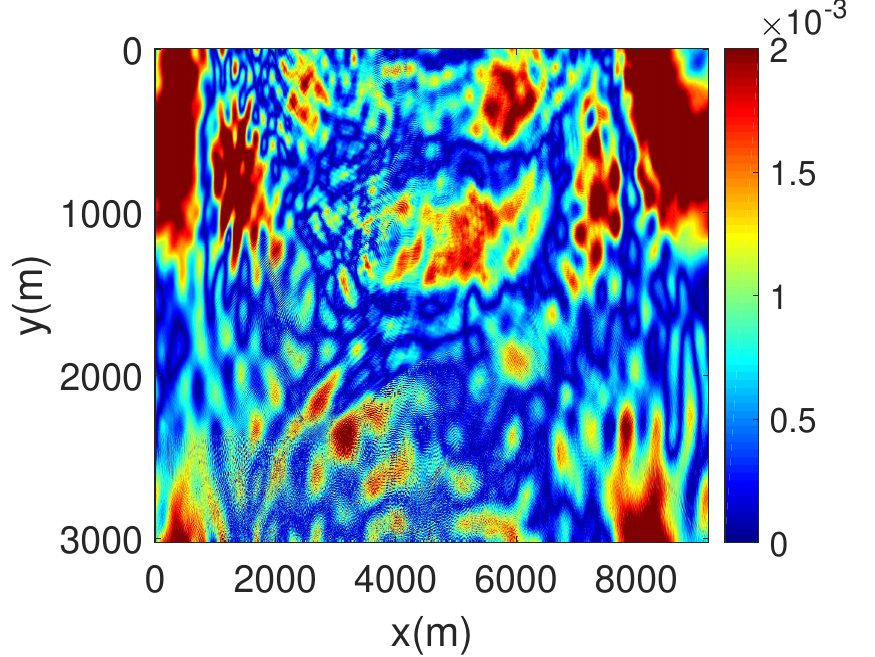} \\
    \includegraphics[width=0.45\linewidth]{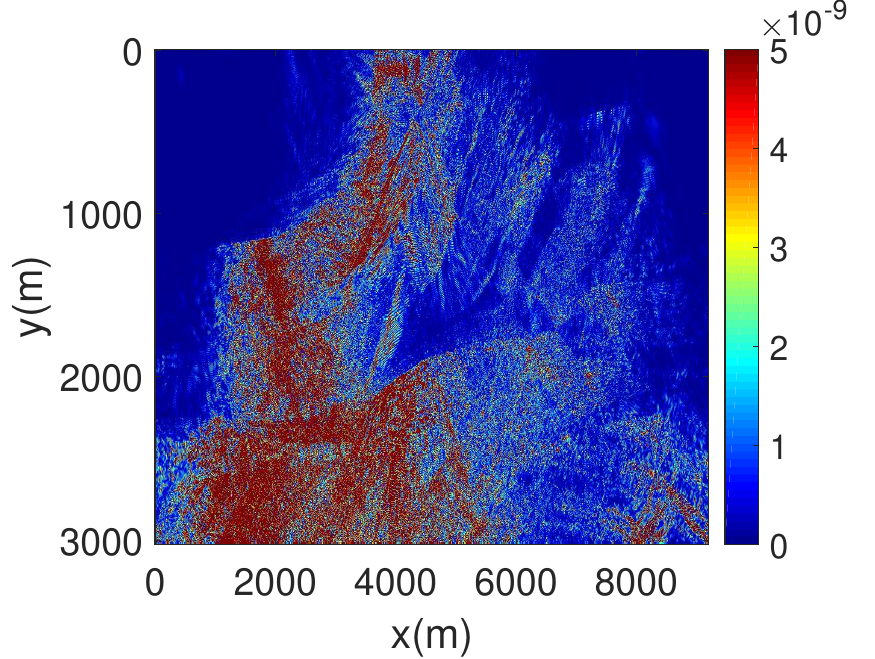}
    \includegraphics[width=0.45\linewidth]{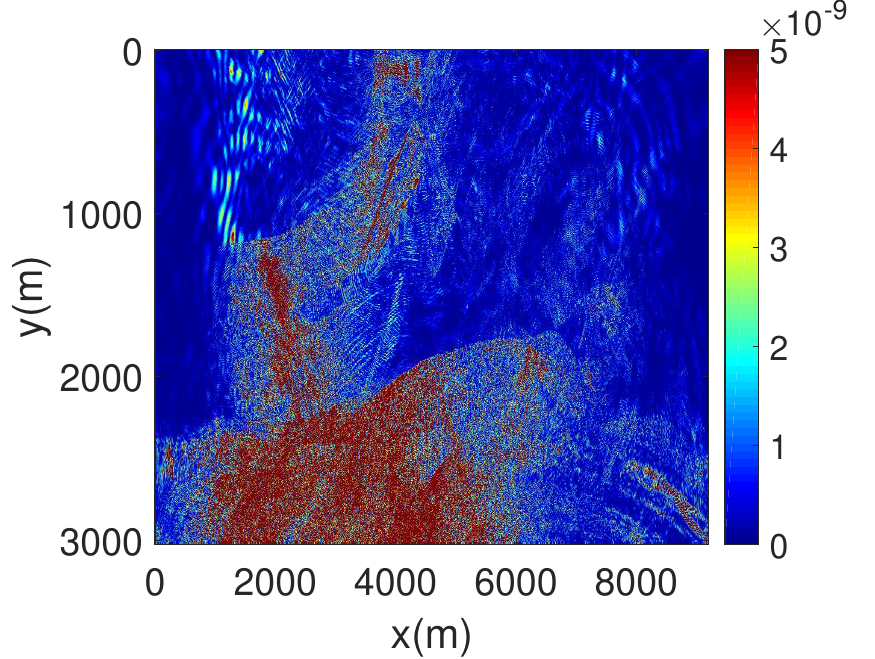}
    \caption{Absolute and energy of absolute error field at $T=2(s)$ when fine and coarse propagators run on the same grid for Marmousi example in Sec~\ref{sec:Marm-ex}. Top row, left: absolute error of wavefield for $k=1$ solution, right: absolute error of wavefield for $k=7$ solution. Bottom row, left: energy error for $k=1$ solution, right: energy error for $k=7$ solution. }
    \label{fig:twoDmarm_samegriderrorfield}
\end{figure}

\begin{figure}
    \centering
    \includegraphics[width=0.45\linewidth]{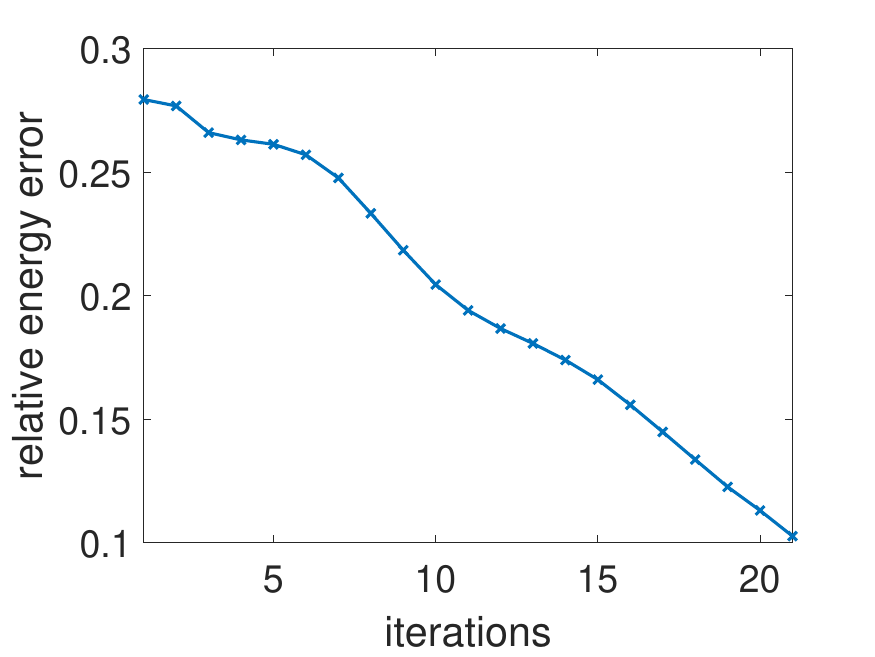}
    \includegraphics[width=0.45\linewidth]{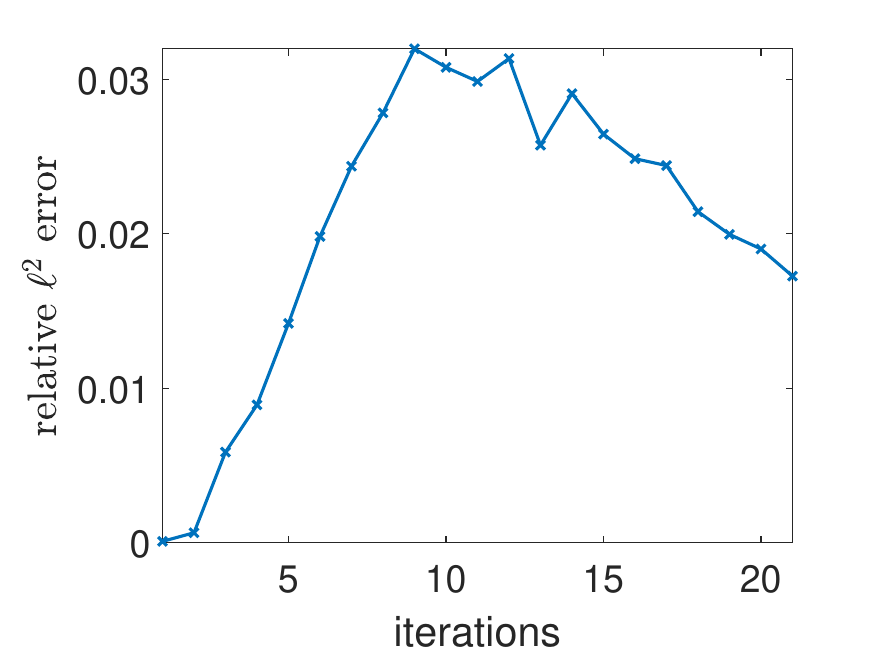}
    \caption{Relative errors when coarse and fine propagators run on the same grid for the Marmousi model. Left: the energy error, right: the $\ell^2$ error.}
    \label{fig:twoDmarm_samegriderror}
\end{figure}

We test our method with the Marmousi wave speed model \cite{brougois1990marmousi}, as shown in Figure \ref{fig:twoDmarm_model}. The fine scale domain has $2422\times 7367$ grid points while the coarse scale has $49 \times 147$ grid points or 50 times smaller in each dimension. The initial data is a pulse waveform centered at $x_0=(400m,3880m),$ where $m$ denotes the length unit in meter, 
\begin{align*}
u(x,y;0) &=\cos(0.01(x-x_0))\exp(-1.6\cdot10^{-5}((x-x_0)^2)),\\
u_t(x,y;0) & =0.
\end{align*}
The discretization parameters are in the following table where coarse and fine computation communicate every $500$ coarse time steps
\bigskip
\begin{center}
\begin{tabular}{ |c|c|c|c|c|c|c|c|c|c|}
\hline
 $T$ & $\Delta t_{com}$ & $\Delta x$ & $\Delta t/\Delta x$ & $\Delta x/\delta x$ &  $\Delta t / \delta t$ & $\mathcal{I}$ & $\nabla_h$ & tol\\ 
 \hline
 $2(s)$ & $0.05(s)$ & $62.45(m)$ & $1.6 \cdot 10^{-4}$ & $50$ & $500$ & \texttt{imresize} & 4 order & $10^{-10}$\\
 \hline
\end{tabular}
\end{center}\bigskip
The computation is executed on one node consisting of $20$ cores on the Stampede2 system at Texas Advanced Computing Center (TACC). 
With our non-optimized MATLAB code, it took 26 hours to run 6 parareal iterations and 12 hours to compute the serial fine solution. Hence each iteration takes about 4 hours, almost 3 times faster on the wall clock than the serial fine computation.
For more detailed experiment, see Section \ref{sec:timing}.

Figure \ref{fig:twoDmarm_solution} shows the solutions computed by the proposed method. 
One observes that some finer details are added back to the computed solution along the iterations. However, Figure~\ref{fig:twoDmarm_error} reveals that the errors decreases rather slowly after the first few iterations.
Indeed, the setup in this experiment is a challenging example of strong scattering due to discontinuities in the wave speed (compared to the previous Example).

It is natural to wonder if the proposed method computes solutions that would converge to
the serially computed fine solutions, \emph{when the coarse and fine propagators run on the same spatial grid.} For this purpose, it suffices to consider a smaller version of the Marmousi velocity model, which
is defined on  $485\times1474$ grid points. 
A different set of discretization parameters are described in the following table
\bigskip
\begin{center}
\begin{tabular}{ |c|c|c|c|c|c|c|c|c|c|}
\hline
 $T$ & $\Delta t_{com}$ & $\Delta x$ & $\Delta t/\Delta x$ & $\Delta x/\delta x$ &  $\Delta t / \delta t$ & $\mathcal{I}$ & $\nabla_h$ & tol\\ 
 \hline
 $2(s)$ & $0.05(s)$ & $6.245(m)$ & $3.2026\cdot 10^{-6}$ & $1$ & $10$ & \texttt{imresize} & 4 order & $10^{-10}$\\
 \hline
\end{tabular}
\end{center}\bigskip
Figure \ref{fig:twoDmarm_samegriderrorfield} shows the absolute error $|u^k_n -u(t_n)|$ and energy error fields at iterations $k=1$ and $k=7$. On the left column, we see that the solution at $k=1$ has larger point-wise absolute error and energy error in regions of high wave speed contrast (e.g. the lower left region in the image domain) than the regions of low wave speed contrast (e.g. upper left region in the image domain). On the right column, however, the solution at $k=7$ has large patches of point-wise absolute error at regions of low wave speed contrast. These errors contribute to the increase of overall $\ell^2$ error in the initial few iterations shown in the right subplot in  Figure~\ref{fig:twoDmarm_samegriderror}. 

We observe a discrepancy between the two errors curves. The energy error decreases while the $\ell^2$ error increases, particularly in the regions of low wave speed contrast. This discrepancy in regions of low wave speed contrast is likely due to the construction of the phase corrector. At regions of high contrast, when locally scattered wave emerged, the phase corrector is constructed to decrease the error there but because it is a global operator, it also perturbs solution everywhere else that in effect increases the overall error. 

\subsubsection{Timing}\label{sec:timing}
 To see how wall clock computing time changes as the number of cores changes, we use the Marmousi model again and the discretization parameters are as follow 
\bigskip
\begin{center}
\begin{tabular}{|c|c|c|c|c|c|}
    \hline
    $T$ & $\Delta t_{com}$ & $\Delta t$ & $\delta t$ & $N_{\Delta x}$ & $N_{\delta x}$  \\
    \hline
    $1,5,10$ & $0.05$ & $8\cdot 10^{-4}$ & $4\cdot 10^{-4}$ & $485\times 1474$ & $2422\times 7367$ \\
    \hline
\end{tabular}
\end{center}
\bigskip
We used an Intel Skylake node and varied the number of cores to perform the computation. In Table \ref{tab:timingparareal}, computing time in seconds is recorded for different parts in the algorithm: parallel computation, creation of the phase corrector (requiring QR), serial coarse update. We see the computing time of the stabilization process is small, relative to other parts of the algorithm. As a benchmark, we also timed the serial fine computation. The projected speed up is calculated as if the number of cores is equal to the number of time slices $n_{CPU} = N = T/\Delta t_{com}$.

\begin{center}
\begin{table}
    \centering
    \caption{Computing time (in seconds) of each part in our algorithm. Number of parareal iteration is 4. Projected speed up is calculated as if the number of CPUs is equal to the number of time slices. In the projected speed up calculation, we assume the time to create phase corrector does not change when the number of CPUs increases.}
    \begin{tabular}{|l|l|l|l|l|l|l|}
    \hline
    Cores & $N=T/\Delta t_{com}$ & \begin{tabular}[c]{@{}l@{}}Parallel \\ computation\end{tabular} & {\color[HTML]{FE0000} \begin{tabular}[c]{@{}l@{}}Creating \\ corrector\end{tabular}} & \begin{tabular}[c]{@{}l@{}}Serial \\ update\end{tabular} & \begin{tabular}[c]{@{}l@{}}Serial fine \\ computation\end{tabular} & \begin{tabular}[c]{@{}l@{}}Projected \\ speed up\end{tabular} \\ \hline
     & 20 & \begin{tabular}[c]{@{}l@{}}186.81\\ 180.84 \\ 180.31 \\ 180.35\end{tabular} & {\color[HTML]{FE0000} \begin{tabular}[c]{@{}l@{}}0.26 \\ 0.11 \\ 0.13 \\ 0.15\end{tabular}} & \begin{tabular}[c]{@{}l@{}}1.99 \\ 1.52 \\ 1.49 \\ 1.52\end{tabular} & 334.49 & \begin{tabular}[c]{@{}l@{}}8.44 \\ 4.32 \\ 2.91 \\ 2.18\end{tabular} \\ \cline{2-7} 
     & 100 & \begin{tabular}[c]{@{}l@{}}908.74 \\ 922.19 \\ 911.85 \\ 921.26\end{tabular} & {\color[HTML]{FE0000} \begin{tabular}[c]{@{}l@{}}0.48 \\ 0.49 \\ 0.64 \\ 0.74\end{tabular}} & \begin{tabular}[c]{@{}l@{}}8.64 \\ 8.46 \\ 8.72 \\ 8.71\end{tabular} & 1724.30 & \begin{tabular}[c]{@{}l@{}}37.92 \\ 18.88 \\ 12.57 \\ 9.40\end{tabular} \\ \cline{2-7} 
    \multirow{-3}{*}{4} & 200 & \begin{tabular}[c]{@{}l@{}}1807.23 \\ 1837.55 \\ 1819.07 \\ 1854.84\end{tabular} & {\color[HTML]{FE0000} \begin{tabular}[c]{@{}l@{}}0.81 \\ 1.06 \\ 3.14 \\ 1.77\end{tabular}} & \begin{tabular}[c]{@{}l@{}}17.13 \\ 17.90\\ 17.82 \\ 18.98\end{tabular} & 3479.63 & \begin{tabular}[c]{@{}l@{}}64.34 \\ 31.69 \\ 20.82 \\ 15.47\end{tabular} \\ \hline
     & 20 & \begin{tabular}[c]{@{}l@{}}119.61 \\ 119.06 \\ 129.25 \\ 119.51\end{tabular} & {\color[HTML]{FE0000} \begin{tabular}[c]{@{}l@{}}0.95 \\ 0.1 \\ 0.11 \\ 0.14\end{tabular}} & \begin{tabular}[c]{@{}l@{}}1.78 \\ 1.41 \\ 1.43 \\ 1.47\end{tabular} & 333.78 & \begin{tabular}[c]{@{}l@{}}7.83 \\ 3.98 \\ 2.60 \\ 1.96\end{tabular} \\ \cline{2-7} 
     & 100 & \begin{tabular}[c]{@{}l@{}}525.09 \\ 547.06 \\ 543.50 \\ 527.34\end{tabular} & {\color[HTML]{FE0000} \begin{tabular}[c]{@{}l@{}}0.46 \\ 0.49 \\ 0.57 \\ 0.73\end{tabular}} & \begin{tabular}[c]{@{}l@{}}8.27 \\ 7.78 \\ 8.31 \\ 8.33\end{tabular} & 1725.16 & \begin{tabular}[c]{@{}l@{}}35.12 \\ 17.34 \\ 11.49 \\ 8.63\end{tabular} \\ \cline{2-7} 
    \multirow{-3}{*}{8} & 200 & \begin{tabular}[c]{@{}l@{}}1023.37 \\ 1050.60 \\ 1032.15 \\ 1029.58\end{tabular} & {\color[HTML]{FE0000} \begin{tabular}[c]{@{}l@{}}0.99 \\ 1.08 \\ 1.39 \\ 1.94\end{tabular}} & \begin{tabular}[c]{@{}l@{}}16.26 \\ 16.53 \\ 17.80 \\ 20.47\end{tabular} & 3491.95 & \begin{tabular}[c]{@{}l@{}}60.02 \\ 29.64 \\ 19.58 \\ 14.43\end{tabular} \\ \hline
     & 20 & \begin{tabular}[c]{@{}l@{}}73.40 \\ 65.93 \\ 63.73 \\ 66.45\end{tabular} & {\color[HTML]{FE0000} \begin{tabular}[c]{@{}l@{}}0.26 \\ 0.11 \\ 0.11 \\ 0.13\end{tabular}} & \begin{tabular}[c]{@{}l@{}}2.41 \\ 1.51 \\ 1.53 \\ 1.58\end{tabular} & 332.83 & \begin{tabular}[c]{@{}l@{}}4.37 \\ 2.32 \\ 1.59 \\ 1.46\end{tabular} \\ \cline{2-7} 
     & 100 & \begin{tabular}[c]{@{}l@{}}337.38 \\ 331.03 \\ 335.75 \\ 345.86\end{tabular} & {\color[HTML]{FE0000} \begin{tabular}[c]{@{}l@{}}0.46 \\ 0.46 \\ 0.58 \\ 0.72\end{tabular}} & \begin{tabular}[c]{@{}l@{}}8.00 \\ 8.18 \\ 8.34 \\ 8.76\end{tabular} & 1734.02 & \begin{tabular}[c]{@{}l@{}}22.84 \\ 11.50 \\ 7.64 \\ 5.67\end{tabular} \\ \cline{2-7} 
    \multirow{-3}{*}{20} & 200 & \begin{tabular}[c]{@{}l@{}}656.80 \\ 644.41 \\ 655.77 \\ 653.80\end{tabular} & {\color[HTML]{FE0000} \begin{tabular}[c]{@{}l@{}}1.06 \\ 1.03 \\ 1.35 \\ 1.73\end{tabular}} & \begin{tabular}[c]{@{}l@{}}16.66 \\ 17.73 \\ 17.39 \\ 19.35\end{tabular} & 3492.86 & \begin{tabular}[c]{@{}l@{}}41.88 \\ 20.96 \\ 13.92 \\ 10.35\end{tabular} \\ \hline
    \end{tabular}
    \label{tab:timingparareal}
\end{table}
\end{center}

\section{Summary and conclusion}
We present here a new stable parareal-like method for the second order wave equation. The method uses the solutions computed along the iterations to construct linear operators which bridge the energy difference between the coarse and fine propagators. Such operators are referred to as the phase correctors in this paper. 
We presented an extensive set of numerical studies which aim at revealing the properties of the
proposed method. From the experiments, we see that the proposed method works well for constant and smooth wave speeds. 

For piece-wise smooth wave speeds, the algorithm is stable, but does not seem to produce numerical solutions that converge to the solutions computed by the fine propagators (as the number of iterations increase), when the fine and coarse propagators run on different spatial resolution. 
This is expected because the higher Fourier modes of the solutions computed by the fine propagator on a finer spatial grid cannot be resolved by coarser grids.
This is true even when the initial wavefield is resolved by the \reviewertwo{coarse grid}.
As our simulations reveal, the stagnation of the errors may be caused additionally by a couple of different approximations used in the algorithm. This paper outline these factors for future improvement.
In the last two examples involving piece-wise smooth wave speeds with high contrast, 
we observe that the relative errors are in general much larger than the previous cases. Most likely, this is due to strong local scattering of waves cause by the discontinuities in the wave speeds. Such scatterings cannot be corrected efficiently by the proposed Procrustean approach. 

\reviewerone{Finally, if domain decomposition in space is applied, due to the finite speed of propagation nature of wave, we expect that different phase correctors in the subdomains can be constructed in the same way and the resulting algorithm would be stable. This important topic should be investigated more carefully in a separate paper.}

\section*{Acknowledgment}
The authors are supported partially by NSF grants DMS-1620396 and DMS-1720171. 
Nguyen is supported by an ICES NIMS fellowship.
Part of this research was performed while the second author was visiting the Institute for Pure and Applied Mathematics (IPAM), which is supported by the National Science Foundation (Grant No. DMS-1440415). This work was partially supported by a grant from the Simons Foundation.
The authors thanks TACC for providing computing resources.
\bibliographystyle{abbrv}
\bibliography{main}

\end{document}